\pdfoutput=1
\documentclass[11pt,a4paper,reqno]{amsart}

\usepackage[margin=1in]{geometry}
\IfFileExists{glyphtounicode.tex}{\input{glyphtounicode}\pdfgentounicode=1}{}
\usepackage{cmap}
\usepackage[T1]{fontenc}
\usepackage[utf8]{inputenc}
\usepackage{lmodern}
\usepackage{microtype}
\usepackage{amsmath,amssymb,amsthm,mathtools,mathrsfs}
\usepackage{xcolor}
\usepackage[hypertexnames=false,colorlinks=true,linkcolor=blue!35!black,
  citecolor=blue!35!black,urlcolor=blue!35!black]{hyperref}
\hypersetup{
  pdftitle={Algebraic Transfer of Gaussian Chaoses: Moment Geometry and Spectral Ideals},
  pdfauthor={Guangqian Zhao},
  pdfsubject={Gaussian partition moments, operator contractions, and spectra of Wick multipliers and random convolutions},
  pdfkeywords={Gaussian chaoses, tensor contractions, Schatten ideals,
    singular Wick multipliers, projected Plancherel measures}
}

\numberwithin{equation}{section}
\allowdisplaybreaks[2]
\theoremstyle{plain}
\newtheorem{theorem}{Theorem}[section]
\newtheorem{proposition}[theorem]{Proposition}
\newtheorem{lemma}[theorem]{Lemma}
\newtheorem{corollary}[theorem]{Corollary}

\theoremstyle{definition}
\newtheorem{definition}[theorem]{Definition}

\theoremstyle{remark}

\newcommand{\R}{\mathbb R}
\newcommand{\C}{\mathbb C}
\newcommand{\N}{\mathbb N}
\newcommand{\E}{\mathbb E}
\newcommand{\Prob}{\mathbb P}

\newcommand{\cC}{\mathcal C}

\newcommand{\cE}{\mathcal E}

\newcommand{\cH}{\mathcal H}

\newcommand{\cL}{\mathcal L}
\newcommand{\cP}{\mathcal P}
\newcommand{\cT}{\mathcal T}
\newcommand{\Id}{\operatorname{Id}}
\newcommand{\rank}{\operatorname{rank}}
\newcommand{\Sch}{\mathfrak S}
\newcommand{\Flat}{\mathcal F}
\newcommand{\Prof}{\mathfrak P}

\newcommand{\op}{\mathrm{op}}
\newcommand{\Tr}{\operatorname{Tr}}

\newcommand{\wick}[1]{:#1:}
\newcommand{\norm}[1]{\left\lVert #1\right\rVert}
\newcommand{\ip}[2]{\left\langle #1,#2\right\rangle}
\newcommand{\conj}[1]{\overline{#1}}
\newcommand{\Dya}{\mathbb D}

\title[Algebraic transfer of Gaussian chaoses]{Algebraic Transfer of Gaussian Chaoses: Moment Geometry and Spectral Ideals}

\author{Guangqian Zhao}

\address{School of Mathematical Sciences, University of Science and
Technology of China, Hefei 230026, China}
\email{zhaoguangqian@mail.ustc.edu.cn}
\date{}

\begin{document}

\begin{abstract}
We prove the partition moment conjecture of Adamczak, Lata\l a and
Meller for decoupled Gaussian chaoses of every order with coefficients in
arbitrary real Banach spaces. The corresponding pairing bound retains
partition weights only on the unpaired slots. We give Talagrand's
requested construction of admissible partitions for unions of operator
ellipsoids and refute Bednorz and Martynek's Bernoulli increment
comparison conjecture. For directed operator contractions, we obtain
constant-one Schatten bounds and the optimal logarithmic loss for
weak Schatten links. Strictly entrywise-positive block matrices refute
Li and Zhao's trace-norm compression conjecture. For radial singular
Wick multipliers on groups of polynomial growth, we determine exact
Schatten thresholds, almost-sure singular-value decay and critical Weyl
constants. The constants come from Gaussian Wick matrices in the
virtually abelian case and powers of a free circular element otherwise.
For finite-group convolutions, we classify projected
Plancherel limits and give singular-value and Gaussian eigenvalue limits
at general coefficient covariance.
\end{abstract}

\maketitle

\section{Gaussian partition moments}
\label{sec:intro}

What determines the $L^p$ norm of a Banach-valued Gaussian chaos?
For a decoupled chaos the answer uses partitions of subsets of its
Gaussian slots, with weight $p^{|\pi|/2}$ for a partition $\pi$.

Let $H_1,\ldots,H_d$ be finite-dimensional real Hilbert spaces, let
$F$ be a real Banach space, and let $A:\bigotimes_{i=1}^d H_i\to F$
be linear, with Hilbert tensor products. Write
$X_A=A(g_1,\ldots,g_d)$ for evaluation on independent standard Gaussian
vectors, and $\norm{X}_p=(\E\norm{X}^p)^{1/p}$.
For $J\subseteq[d]=\{1,\ldots,d\}$ let $\cP(J)$ be the set of
partitions of $J$. For $\pi\in\cP(J)$, the notation $A(x_\pi,g_{J^c})$
means evaluation on
$\bigotimes_{B\in\pi}x_B\otimes\bigotimes_{i\notin J}g_i$ after
permuting the factors back to their original slots. Set
\[
 a_\pi(A)=\sup_{\substack{x_B\in\bigotimes_{i\in B}H_i\\\norm{x_B}_2\le1\ (B\in\pi)}}
 \E\norm{A(x_\pi,g_{J^c})},\qquad
 R_p(A)=\sum_{J\subseteq[d]}\sum_{\pi\in\cP(J)}p^{|\pi|/2}a_\pi(A).
\]
The empty partition contributes $a_\varnothing(A)=\E\norm{X_A}$.
For $U\subseteq[d]$, let $R_p(A;U)$ be the sum restricted to
$J\subseteq U$. For order zero these norms are the coefficient norm.

\begin{theorem}[Partition moment estimate]\label{thm:banach-partition}
For every $d\ge1$ and $1\le p<\infty$,
\begin{equation}\label{eq:partition-moment-bound}
 C_d^{-1} R_p(A)\le \norm{X_A}_p\le C_d R_p(A),
\end{equation}
where $C_d$ depends only on $d$.
\end{theorem}

The upper estimate answers \cite[Conjecture~2.2]{ALM}. The lower estimate
is their Theorem 2.1. The quadratic case follows from the matrix
width bound of Bednorz, Martynek and Meller \cite{BMM}. For higher orders
we apply their Proposition~3.1 to successive one-coordinate replacements.
The proof below uses finite coefficient
spans, including in the auxiliary Banach spaces. We use the standard
linear Gaussian estimate for a deterministic map $T:H\to E$ into any
Banach space,
\begin{equation}\label{eq:partition-linear}
 \norm{Tg}_p\asymp \E\norm{Tg}+\sqrt p\,\norm{T}_{\op}.
\end{equation}
For a random $T$ independent of $g$, applying this conditionally yields the same comparison with the $L^p$ norms of both terms on the right (the first is a conditional mean). We also use $\norm{Z}_2\lesssim_r\norm{Z}_1$ for any Banach-valued decoupled chaos of order $r$, including seminorms. This follows by applying the one-vector Gaussian moment comparison successively and using Minkowski's integral inequality.
For distinct $i,j$ define, with unit-ball suprema in the displayed Hilbert variables,
\begin{align*}
 W_i&=\E_{-i}\sup_{x_i}\norm{A(x_i,g_{-i})},\\
 W_{ij}&=\E_{-\{i,j\}}\sup_{x_i,x_j}\norm{A(x_i,x_j,g_{-\{i,j\}})},\\
 B_{ij}&=\E_{-\{i,j\}}\sup_{x_i}\E_j\norm{A(x_i,g_j,g_{-\{i,j\}})}.
\end{align*}
Here $\E_{-I}$ integrates the Gaussian coordinates outside $I$. Put $S=\E\norm{X_A}$ and $H=a_{\{[d]\}}(A)$, the norm of the linear map on the full Hilbert tensor product.
\begin{lemma}\label{lem:partition-width}
For $d\ge2$,
\begin{equation}\label{eq:partition-width}
 W_i^2\lesssim_d H^2+\sum_{j\ne i}(S W_{ij}+W_i B_{ij}).
\end{equation}
\end{lemma}
\begin{proof}
For a bounded family $\mathcal M$ of real matrices $M:H_i\to H_j$, write
\begin{align*}
 D_*&=\sup_{M\in\mathcal M}\norm{M}_{\op},&
 S_*&=\E_{u,h}\sup_{M\in\mathcal M}|\ip{u}{Mh}|,\\
 B_*&=\sup_{\norm{x}\le1}\E_u\sup_{M\in\mathcal M}|\ip{u}{Mx}|,&
 W_*&=\E_u\sup_{M\in\mathcal M}\norm{M^*u}.
\end{align*}
Here $u,z$ are independent standard Gaussians in $H_j$ and $h$ is standard
in $H_i$, independent of them. The correlated quadratic comparison of
Bednorz, Martynek and Meller \cite[Prop.~3.1]{BMM} is
\begin{equation}\label{eq:partition-quadratic-comparison}
 \E_{u,z}\sup_{M\in\mathcal M}|\ip{z}{MM^*u}|
 \lesssim D_* S_*+B_* W_* .
\end{equation}
Finite families suffice by increasing approximation from a countable dense
subset of the closure of $\mathcal M$.

For $\varphi\in B_{F^*}$ let $Z_\varphi(g_{-i})\in H_i$ represent $x\mapsto\varphi(A(x,g_{-i}))$, and set $f_\varphi=\norm{Z_\varphi}^2$. Duality and Gaussian orthogonality give
\[
 W_i=\E\sup_\varphi\norm{Z_\varphi},\qquad
 \sup_\varphi\E f_\varphi=\sup_\varphi\norm{\varphi\circ A}_{\mathrm{HS}}^2=H^2.
\]
If $g'_{-i}$ is an independent copy, Jensen's inequality gives
$W_i^2\le H^2+\E\sup_\varphi(f_\varphi(g_{-i})-f_\varphi(g'_{-i}))$.
Replace coordinates of $g_{-i}$ by those of $g'_{-i}$ one at a time and telescope. Each hybrid tuple has the original product law, and thus
\begin{equation}\label{eq:partition-width-telescoping}
 W_i^2\le H^2+\sum_{j\ne i}\E\sup_\varphi
 |f_\varphi(g_{-i})-f_\varphi(g_{-i}^{(j)})|,
\end{equation}
where $g_{-i}^{(j)}$ replaces only coordinate $j$ by an independent copy.

Fix $g_{-\{i,j\}}$. The matrix $M_\varphi:H_i\to H_j$ defined by
$\ip{y}{M_\varphi x}=\varphi(A(x,y,g_{-\{i,j\}}))$ gives
$Z_\varphi(g_{-i})=M_\varphi^*g_j$ and
$Z_\varphi(g_{-i}^{(j)})=M_\varphi^*g'_j$.
For the independent standard Gaussians
$u=(g_j+g'_j)/\sqrt2$ and $z=(g_j-g'_j)/\sqrt2$,
\[
 f_\varphi(g_{-i})-f_\varphi(g_{-i}^{(j)})
 =2\ip{z}{M_\varphi M_\varphi^*u}.
\]
By \eqref{eq:partition-quadratic-comparison} the conditional expected
supremum of the absolute difference is at most
$C(D_* S_*+B_* W_*)$ for this family. These four quantities are
seminorms of the remaining coefficient tensor, with respective
expectations $W_{ij},S,B_{ij},W_i$. Cauchy--Schwarz and the
order-$d-2$ moment comparison therefore give
\[
 \E(D_* S_*+B_* W_*)\lesssim_d S W_{ij}+W_i B_{ij}.
\]
Inserting this into \eqref{eq:partition-width-telescoping} proves \eqref{eq:partition-width}. All suprema may be taken over compact sets in finite coefficient spans, so no measurability issue arises.
\end{proof}
\begin{proof}[Proof of Theorem~\ref{thm:banach-partition}]
Assume the upper estimate in \eqref{eq:partition-moment-bound} for all orders below $d$ and all Banach spaces. The case $d=1$ is \eqref{eq:partition-linear}. Fix $p\ge1$, and abbreviate $t=\sqrt p$, $R=R_p(A)$. We have $S\le R$ and $H\le R/t$.

For each $j$ let $\gamma_j$ be the standard Gaussian law on $H_j$.
Define two order-$d-1$ chaoses by
\[
 Y_j:y\mapsto A(g_{-j},y)\quad\text{in }L^1(H_j,\gamma_j;F),
 \qquad
 V_j:x\mapsto A(g_{-j},x)\quad\text{in }\cL(H_j,F).
\]
For $Y_j$ every partition seminorm equals the corresponding $a_\pi(A)$
with $j$ outside its support. The induction hypothesis gives
$\norm{Y_j}_p\lesssim_d R$. Conditional application of
\eqref{eq:partition-linear} to its $i$th coordinate gives
\begin{equation}\label{eq:partition-conditional-mean}
 t B_{ij}\lesssim_d R\qquad(i\ne j).
\end{equation}
For $V_j$ the empty-partition term is $W_j$. Write $b_\pi=a_\pi(V_j)$, for $\pi$ on $J\subseteq[d]\setminus\{j\}$. If $J\ne\varnothing$, fix all testing tensors $x_B$ for $B\in\pi$ and denote the contracted tensor on $[d]\setminus J$ by $A_x$. The remaining order $d-|J|$ lies between $1$ and $d-1$; in particular it still contains $j$. Applying \eqref{eq:partition-linear} conditionally and then the induction hypothesis gives
\[
 t\,\E\sup_{x_j}\norm{A_x(x_j,g_{-(J\cup\{j\})})}
 \lesssim_d\norm{X_{A_x}}_p\lesssim_d R_p(A_x).
\]
Each partition $\rho$ of a subset of $[d]\setminus J$ contributes
to the term indexed by $\pi\cup\rho$ in $R$. Thus for each choice of
the fixed tests, $t^{|\pi|}R_p(A_x)\le R$. Taking the supremum over
those tests gives $t^{|\pi|+1}b_\pi\lesssim_d R$.
Summing over the finitely many $\pi$ (a number depending only on $d$), and applying the induction hypothesis to $V_j$, yields
\begin{equation}\label{eq:partition-conditional-operator}
 \norm{V_j}_p\lesssim_d W_j+R/t.
\end{equation}
Another conditional application of \eqref{eq:partition-linear}, now in coordinate $i$ of $V_j$, gives $t W_{ij}\lesssim\norm{V_j}_p$, whence
\begin{equation}\label{eq:partition-two-legs}
 W_{ij}\lesssim_d W_j/t+R/t^2.
\end{equation}
Take $M=\max_i W_i$. By \eqref{eq:partition-width}, \eqref{eq:partition-conditional-mean}, \eqref{eq:partition-two-legs}, and the bounds on $S,H$,
\[
 M^2\lesssim_d R^2/t^2+MR/t,
 \qquad\text{hence}\qquad M\lesssim_d R/t.
\]
Finally, condition on all variables except $g_i$ in \eqref{eq:partition-linear}, and use \eqref{eq:partition-conditional-operator} and the estimate for $Y_i$:
\[
 \norm{X_A}_p\lesssim \norm{Y_i}_p+t\norm{V_i}_p
 \lesssim_d R+tW_i\lesssim_d R.
\]
This closes the induction and proves \eqref{eq:partition-moment-bound}.

\end{proof}

\begin{corollary}[Banach-valued Gaussian polynomials]
\label{cor:banach-polynomial-partition}
Let $f:\R^n\to F$ be a polynomial of degree at most $D$ and let $G$ be
standard Gaussian. For $1\le d\le D$, write $A_d=\E\nabla^d f(G)$ and
$M_p(f)=\sum_{d=1}^{D}R_p(A_d)$. For $1\le p<\infty$,
\[
 \norm{f(G)-\E f(G)}_{L^p(F)}\asymp_D M_p(f).
\]
If $f(G)$ is nonconstant, there are $c_D,C_D>0$ such that
\[
 \Prob\!\left(\norm{f(G)-\E f(G)}>C_D M_p(f)\right)\le e^{-p},
 \qquad
 \Prob\!\left(\norm{f(G)-\E f(G)}>c_D M_p(f)\right)\ge e^{-C_Dp}.
\]
\end{corollary}

\begin{proof}
Combine Theorem~\ref{thm:banach-partition} with \cite[Prop.~2.11]{ALM},
which compares the polynomial moment with the sum of the moments of the
decoupled chaoses with kernels $A_d$. Markov's inequality at moment $p$
gives the upper tail. Paley--Zygmund applied to the $p$th power of the norm
gives the lower tail, since $M_{2p}(f)\le 2^{D/2}M_p(f)$.
This gives the removal of the Banach-space hypothesis in
\cite[Remark~2.12]{ALM}.
\end{proof}

\section{Gaussian pairings and increment comparison}
\label{sec:gaussian-pairings}

A common tensor index can be summed by evaluating both slots at the
same standard Gaussian and averaging. Cauchy--Schwarz on each edge
controls such a pairing and leaves partition tests on the free slots.

\subsection{Paired tensor contractions}

At each vertex $v$ of a finite graph let $A_v$ be a finite tensor
with real Hilbert slots and coefficients in a real Banach space $F_v$. Pair slots at distinct
vertices, at most once per slot, identifying the two Hilbert spaces
at each pair. Multiple pairs between the same vertices and cycles are
allowed. Let $U_v$ be the free slots. Given a bounded multilinear
$b:\prod_v F_v\to F$, write $C$ for the tensor obtained by summing common
orthonormal basis indices on the paired slots and applying $b$ to the
coefficients in vertex order.

\begin{theorem}[Partition bounds under paired contractions]
\label{thm:partition-contractions}
For $1\le p<\infty$,
\begin{equation}
\label{eq:partition-contraction-bound}
 R_p(C)\le c\norm b\prod_v R_p(A_v;U_v),
\end{equation}
where $c$ depends only on the input orders.
For $p\ge2$ the full decoupled evaluation norms satisfy the
order-independent inequality
\begin{equation}
\label{eq:gaussian-evaluation-contraction}
 \norm{X_C}_p\le\norm b\prod_v\norm{X_{A_v}}_p.
\end{equation}
\end{theorem}

\begin{proof}
Use a single independent Gaussian $z_e$ at both slots of pair $e$ and
independent Gaussians $G_v$ on every free slot. Write
$Z_v=A_v(G_v,(z_e)_{e\ni v})$. The covariance identity gives
$X_C=\E_z b((Z_v)_v)$. With the free variables fixed, Cauchy--Schwarz
in each $z_e$ acts on the two endpoint factors. The subsequent
$L^2$ integrations at each vertex combine by Fubini. This is the
graph case of Finner's inequality \cite{Finner}. Taking the $L^p$
norm in the disjoint free families gives
\[
 \norm{X_C}_p\le\norm b
 \prod_v\norm{(\E_{(z_e)_{e\ni v}}\norm{Z_v}^{\,2})^{1/2}}_{L^p(G_v)}.
\]
For $p\ge2$, Jensen bounds every factor by $\norm{X_{A_v}}_p$.
For \eqref{eq:partition-contraction-bound}, fixed-order Gaussian moment
comparison in the paired variables bounds each factor, up to an
order-dependent constant, by
$\norm{\E_{(z_e)_{e\ni v}}\norm{Z_v}}_{L^p(G_v)}$.
Theorem~\ref{thm:banach-partition} in the value space
$L^1((z_e)_{e\ni v};F_v)$ identifies this norm up to constants with
$R_p(A_v;U_v)$, since each partition test acts on free variables.
The lower estimate of that theorem for $C$ finishes the proof.
\end{proof}

One may also insert a bounded real connector on each matched pair, with
an additional factor equal to the product of connector norms. For a
connector $U_e$ between different endpoint spaces, use $z_e$ in its source
and $U_ez_e$ in its target. To see the bound,
normalize each nonzero connector to have norm one or less and absorb it
into one of the two slots. In that slot the resulting Gaussian has
covariance at most the identity. Adding independent Gaussian noise with
the complementary covariance and applying conditional Jensen shows that
the full evaluation norm and every restricted partition term can only
increase on restoring the standard covariance. Only averaged slots are
changed in the restricted partition terms.

Both estimates extend by continuity from algebraic tensors to the
corresponding completions on fixed Hilbert spaces. The mean term makes $R_p(\cdot;U)$ a norm:
orthogonality of scalar basis monomials recovers every coefficient from
the decoupled evaluation. Expanding changes one input at a time in
\eqref{eq:partition-contraction-bound} or
\eqref{eq:gaussian-evaluation-contraction} therefore defines each
completed contraction. Pairing two slots at the same vertex would need
additional control: the scalar identity matrix in dimension $N$ has
self-contraction $N$ and decoupled $L^2$ norm $\sqrt N$.

\paragraph{Trace-class branches.}
Take $b(S,T)=ST$ from $\mathcal S_2\times\mathcal S_2$ to $\mathcal S_1$,
with compatible operator spaces. Its norm is at most one by Schatten
H\"older. Complex coefficient spaces are viewed as real Banach spaces.
For symmetric kernels $A,B$ of orders $m,n$ on one Gaussian Hilbert
space, the term of degree $m+n-2r$ in $I_m(A)I_n(B)$ is
$r!\binom mr\binom nr\, I_{m+n-2r}(\operatorname{Sym}(A\otimes_r B))$,
with $I_k$ the usual multiple Wiener integral. The scalar Wick product
formula \cite{Janson,Nualart} applies coefficientwise, keeping the order
of the operator products. Symmetrization is contractive for $R_p$ by
permutation invariance and the triangle inequality. Fixed-order Gaussian decoupling
compares the $L^p$ norms of the symmetric evaluation and its decoupled
version with constants depending only on the order, uniformly in $p$
\cite[Section~5.2]{ALM}. Thus each term, including the degree-zero term,
has $L^p(\mathcal S_1)$ norm bounded by
\[
 c_{m,n}\, R_p(A;[m-r])\,R_p(B;[n-r]),
\]
where symmetry allows the free slots to be numbered first on each input.
This bound gives termwise convergence in $L^p(\mathcal S_1)$ under
approximation in the displayed input norms. Approximation in the full
input norms gives convergence of the ordinary product too, since
$R_{2p}(A)\le2^{m/2}R_p(A)$ and likewise for $B$.

Let $g,h$ be independent standard Gaussians in $\R^N$ and let $z$ be an
independent scalar Gaussian. The column $zg$ and row $zh^\mathsf{T}$ have
product $(z^2-1)gh^\mathsf{T}+gh^\mathsf{T}$, of degrees four and two.
In the labelled kernels the $z$ slots are paired to yield $gh^\mathsf{T}$.
Its trace-norm $L^p$ moment is comparable to
$(\sqrt N+\sqrt p)^2$. Each restricted profile has order
$\sqrt N+\sqrt p$, while each full profile has order
$\sqrt p(\sqrt N+\sqrt p)$. At the coefficient level,
the contracted tensor has entries $e_i e_j^\mathsf{T}$. The oriented cut
$e_i\otimes e_k\mapsto\sum_j e_j\otimes e_i e_j^\mathsf{T}e_k
=e_k\otimes e_i$ has trace norm $N^2$, while each input cut has
Hilbert--Schmidt norm $\sqrt N$. This gives the dimensional obstruction
to a mixed $\mathcal S_2\times\mathcal S_2\to\mathcal S_1$ bound on all
oriented cuts.

\subsection{Chaining for unions of ellipsoids}
The Gaussian width $W_*$ in the proof of Lemma~\ref{lem:partition-width}
is indexed by the union of the ellipsoids $M B_{H_i}$,
$M\in\mathcal M$. For such unions, Talagrand's
Proposition~15.1.21 and the majorizing measure theorem give a bound
for $\gamma_2$ via Gaussian widths. Research Problem~15.1.22
of \cite{TalagrandULB} asks for a direct construction of admissible
partitions with this bound. For a single ellipsoid see
\cite[Section~2.14]{TalagrandULB}. For a union we choose the singular
subspaces along the nested partitions of the operator family.

Recall that an admissible sequence $(\mathcal A_n)_{n\ge0}$ of
partitions of a metric space $(T,d)$ is nested, with
$\mathcal A_0=\{T\}$ and $\#\mathcal A_n\le 2^{2^n}$.
Its cost is $\sup_{t\in T}\sum_{n\ge0}2^{n/2}
\operatorname{diam}_d\mathcal A_n(t)$, where $\mathcal A_n(t)$ is the
cell containing $t$. The infimum of these costs is $\gamma_2(T,d)$.

\begin{theorem}[Direct chaining of operator ellipsoids]
\label{thm:ellipsoid-chaining}
Let $T$ be a nonempty family of Hilbert--Schmidt operators
$t:H_1\to H_2$ between real Hilbert spaces and let
\[
 H=\{t^*x:t\in T,\ \norm x_{H_2}\le1\},\qquad
 V=\sup_{t\in T}\norm t_{\mathrm{HS}}.
\]
Given admissible partitions of $T$ with cost $G$ in the operator
norm, one can construct admissible partitions of $H$ with cost
at most $C(V+G)$ in the Hilbert norm, for an absolute constant $C$.
In particular,
\[
 \gamma_2(H,\norm{\cdot}_{H_1})
 \le C\bigl(V+\gamma_2(T,\norm{\cdot}_{\op})\bigr).
\]
\end{theorem}

\begin{proof}
Assume $V+G<\infty$, with $V>0$ since $H=\{0\}$ otherwise.
Choose a representative $t_n$ of each cell of $\mathcal A_n$, and
use the representatives along the cells containing $t$ to write
$A=t^*$, $A_n=t_n^*$ and $\delta_n=\norm{A-A_n}_{\op}$.
Then $\sum_n2^{n/2}\delta_n\le G$.

Let $P_n$ be the span of the first $2^j$ right singular vectors of
$A_j$ for all $j\le n$, taking only nonzero singular directions if
there are fewer. Put $P_{-1}=\{0\}$ and $E_k=P_k\ominus P_{k-1}$.
Then $\dim E_k\le2^k$. With $b_k=\norm{A|_{E_k}}_{\op}$, the
singular-value perturbation inequality gives
\[
 b_k\le \norm{A|_{P_{k-1}^\perp}}_{\op}
 \le s_{2^{k-1}+1}(A)+2\delta_{k-1}\qquad(k\ge1).
\]
In particular $A$ vanishes on the orthogonal complement of the
closed span of all the $E_k$. For $x$ in the unit ball write $x_k$
for its projection on $E_k$ and $r_k=\norm{x_k}$. Then
$Ax=\sum_k Ax_k$ and $\sum_k r_k^2\le1$.

For each cell at level $n$, its ancestors determine the spaces
$E_k$, $k\le n$. In each $E_k$, take the union of $\{0\}$ and
nets of the balls of radii $2^{-j}$ for $0\le j\le n$, with mesh
$2^{k-n}2^{-j}$ and centers in the respective balls. Form $S_n$
from the points $A_n\sum_{k\le n}v_k$ for all these choices $v_k$
and all level-$n$ cells. The Euclidean volume bound for the nets
gives
\[
 \log_2\#S_n
 \le 2^n+\sum_{k=0}^n
 \left(\log_2(n+2)+2^k\log_2(1+2^{n-k+1})\right)
 \le C2^n .
\]
For $h=Ax\in H$, use $v_k=0$ when $r_k<2^{-n}$ and otherwise
approximate $x_k$ in a dyadic ball of radius in $[r_k,2r_k]$.
The orthogonality of the $E_k$ gives $\norm{\sum_{k\le n}v_k}\le2$.
Writing $e_n(h)=d(h,S_n)$ yields
\[
 e_n(h)\le 2\delta_n+\sum_{k>n}b_kr_k+
 \sum_{k\le n} b_k(2\cdot 2^{k-n}r_k+2^{-n}).
\]
Summing with weights $2^{n/2}$ and exchanging the sums gives
\[
 \sum_{n\ge0}2^{n/2}e_n(h)
 \le C\left(G+\sum_{k\ge0}2^{k/2}b_kr_k+
                  \sum_{k\ge0}2^{-k/2}b_k\right)
 \le C'(V+G).
\]
For the last step, $b_k\le V$ for all $k$. For the middle sum,
the perturbation terms cost at most $C G$, and Cauchy--Schwarz
and monotonicity of singular values give
$\sum_{k\ge1}2^ks_{2^{k-1}+1}(A)^2\le4\norm A_{\mathrm{HS}}^2$.

Finally, for each $n$ group the points of $H$ by their nearest
center in $S_n$ (with ties broken in a fixed order) and by the
dyadic bin of $e_n(h)$, with a bottom bin $e_n(h)\le V2^{-n}$.
There are at most $n+1$ bins since $0\in S_n$ and $e_n(h)\le V$.
Each cell containing $h$ has diameter at most
$4e_n(h)+2V2^{-n}$. Refine these partitions jointly through level
$n$ to make them nested. Their log cardinalities remain bounded
by $C''2^n$. Shifting all levels by a fixed integer and starting
with the one-cell partition gives exact admissibility. The shift
changes the cost by at most an absolute factor and adds $O(V)$
for the initial levels. Taking the infimum over input partitions
proves the $\gamma_2$ bound.
\end{proof}

\subsection{Comparison of Bernoulli increments}
For independent Rademacher variables $\varepsilon_i$ write
$B_t=\sum_i t_i\varepsilon_i$ and
$b(T)=\E\sup_{t\in T}B_t$. Bednorz and Martynek
\cite[Conjecture~2]{BMContraction} asked whether contraction of every
$L^p$ increment for $p\ge2$ gives universal control of $b$.

\begin{theorem}[Failure of Bernoulli moment comparison]
\label{thm:bernoulli-increment-comparison}
For every integer $L\ge2$ there is a finite set $T_L\subset\ell^2$
of finitely supported vectors and a map $\varphi_L:T_L\to\ell^2$
such that for all $p\ge2$ and $t,s\in T_L$,
\[
 \norm{B_{\varphi_L(t)}-B_{\varphi_L(s)}}_p
 \le \norm{B_t-B_s}_p ,
 \qquad
 b(T_L)\le C,\qquad b(\varphi_L(T_L))\ge c\log(1+L),
\]
where $c,C>0$ are absolute.
\end{theorem}

\begin{proof}
Use a rooted tree of depth $L$ with $2^{2m_j}$ children at level $j$,
where $m_j=2^j$, and let $M=2^L$. Assign distinct sign coordinates
to each leaf and two disjoint blocks of $M$ and $m_j$ coordinates to
each level-$j$ edge. For a leaf $u$, let $t_u$ have coefficient one
at its leaf coordinate and coefficient $2^{-j}/\sqrt M$ on the
$M$-block of each level-$j$ edge on its path. Let $w_u$ have
coefficient $1/(jm_j)$ on the corresponding $m_j$-blocks.
All unlisted coefficients are zero. Put $T_L=\{t_u:u\text{ a leaf}\}$.

There are $\prod_{i\le j}2^{2m_i}$ edges at level $j$, with logarithm
of this number bounded by $C2^j$. The normalized sum of the $M$ signs
on each edge is subgaussian with variance proxy one. The expected
maximum of their absolute values is therefore at most $C2^{j/2}$.
Taking the sum over the path levels yields
$b(T_L)\le1+C\sum_{j\ge1}2^{-j/2}$.
For the $w_u$, choose a path by maximizing at each step over the
children of the current vertex. At level $j$ the probability of no
all-positive child block is at most $\exp(-2^{m_j})$. The maximal
edge contribution is $1/j$ on its complement and at least $-1/j$
always. These bounds hold conditionally on the earlier choices, so
$b(\{w_u\})\ge c\sum_{j=1}^L j^{-1}$.

For distinct leaves $u,v$ first separating at level $\ell$, let
$p_*=\min(p,M)$. The triangle inequality and the subgaussian and
uniform bounds give
\[
 \norm{B_{w_u}-B_{w_v}}_p
 \le C\sum_{j=\ell}^L \frac1j\min(1,\sqrt{p_*}\,2^{-j/2}).
\]
Here clipping $p$ at $M$ leaves the minimum unchanged since $M\ge2^j$.
The terms with $j\le2\ell$ have bounded sum, and the rest sum to at
most $C2^{-\ell}\sqrt{p_*}$.

For the input increment, retain either one leaf coordinate or one
level-$\ell$ edge block by conditional Jensen. For $S_M$ the sum of
$M$ independent signs,
$\norm{S_M/\sqrt M}_p\ge c\min(\sqrt p,\sqrt M)$.
Indeed, for $k=\lfloor\min(p,M)/2\rfloor$ the distinct paired-index
terms give $\E S_M^{2k}\ge\binom Mk(2k)!/2^k\ge(Mk/4)^k$.
Thus $\norm{B_{t_u}-B_{t_v}}_p\ge
c(1+2^{-\ell}\min(\sqrt p,\sqrt M))$ by taking the larger of the
two conditional bounds. Defining $\varphi_L(t_u)=a w_u$ for a
sufficiently small absolute $a>0$ proves all the assertions.
\end{proof}

\section{Oriented contractions and spectral ideals}
\label{sec:links-wick}

An input--output cut turns a tensor contraction into an operator
estimate. We start with one shared Hilbert space and then treat
directed graphs, where a feedback edge set carries the cost of
closing cycles. The Ky Fan comparison gives the sharp logarithmic
loss at the weak Schatten endpoint.

\subsection{Schatten bounds for a link}
\label{sec:strong-schatten-links}

For Hilbert spaces $A,B,X,Y,Z$ and bounded operators
$K:X\to A\otimes Z$ and $L:B\otimes Z\to Y$, set
\begin{equation}
\label{eq:link-definition}
 \operatorname{Link}_Z(L,K)
 =(\Id_A\otimes L)\Sigma_{A,B}(\Id_B\otimes K):
 B\otimes X\longrightarrow A\otimes Y,
\end{equation}
where $\Sigma_{A,B}$ interchanges $B$ and $A$.
In orthonormal bases $(a_\alpha)$ and $(b_\beta)$, write
$Kx=\sum_\alpha a_\alpha\otimes K_\alpha x$ and
$L(b_\beta\otimes z)=L_\beta z$. Then
\begin{equation}
\label{eq:link-blocks}
 \operatorname{Link}_Z(L,K)(b_\beta\otimes x)
 =\sum_\alpha a_\alpha\otimes L_\beta K_\alpha x.
\end{equation}

\begin{theorem}[Schatten link theorem]
\label{thm:schatten-link}
For $1\le r\le\infty$ the link satisfies
\begin{equation}
\label{eq:schatten-link}
 \norm{\operatorname{Link}_Z(L,K)}_{\Sch_r}
 \le \norm L_{\Sch_r}\norm K_{\Sch_r}.
\end{equation}
Here $\Sch_\infty$ denotes bounded operators. For $0<r<1$ and
$1\le d:=\dim Z<\infty$, the optimal constant in
\eqref{eq:schatten-link} is $d^{1/r-1}$.
\end{theorem}

\begin{proof}
The operator-norm bound is immediate from \eqref{eq:link-definition}.
For rank-one factors write
$K=\theta_{u,\xi}$ and $L=\theta_{\eta,v}$, with
$\theta_{u,\xi}x=\ip{x}{\xi}u$. Identify $u\in A\otimes Z$ with
$\Gamma_u\in\Sch_2(\conj Z,A)$, determined by
$\Gamma_u(\conj z)=(\Id_A\otimes\ip{\cdot}{z})u$,
and similarly identify $v\in B\otimes Z$ with $\Gamma_v$.
Then
\begin{equation}
\label{eq:rank-one-link-factorization}
 \operatorname{Link}_Z(\theta_{\eta,v},\theta_{u,\xi})
  =(\Gamma_u\Gamma_v^*)\otimes\theta_{\eta,\xi}.
\end{equation}
On simple tensors this identity follows from \eqref{eq:link-blocks}.
For general $u,v$, algebraic tensor approximations converge in the
Hilbert norms, hence in $\Sch_2$ under $\Gamma$. Schatten--H\"older
gives convergence of the product in $\Sch_1$, and the operator-norm
bound identifies the limit in \eqref{eq:rank-one-link-factorization}.
In particular its trace norm is at most
$\norm u\norm v\norm\eta\norm\xi$. Summing the singular expansions
gives \eqref{eq:schatten-link} for $r=1$ and finite-rank inputs.
Bilinear complex interpolation between the trace and operator endpoints
\cite{BerghLofstrom} gives $1<r<\infty$ on finite-dimensional supports.

For $0<r<1$, the first factor in \eqref{eq:rank-one-link-factorization}
has rank at most $d$, so its $\Sch_r$ quasi-norm is at most
$d^{1/r-1}$ times its trace norm. Apply the $r$-triangle inequality
for Schatten quasi-norms \cite{SimonTraceIdeals} to the sum of the
rank-one links in the singular expansions. The $r$th powers of the
input singular values separate into a product of sums. This proves
the asserted bound for finite-rank inputs.

For any finite $r>0$, approximating the singular expansions by finite
sums and the tensor singular vectors by finite-support tensors gives
approximants in the input $\Sch_r$ (quasi-)norms. The operator-norm
bound gives convergence of their links in operator norm. The links of
these approximants are finite rank, so the limit is compact and lower
semicontinuity of the singular-value sums gives the same bound
in $\Sch_r$. One-dimensional factors attain constant one for $r\ge1$.
For $r<1$, take $X=Y=\C$, $A=B=Z=\C^d$ and
$u=v=d^{-1/2}\sum_{j=1}^d e_j\otimes e_j$.
Both inputs are rank one of norm one and the link is $d^{-1}\Id$.
\end{proof}

In particular, when both input indices are at most one, $\Sch_1$
is the strongest uniform Schatten target over all contracted
Hilbert spaces. More generally, for arbitrary input indices
$0<r_1,r_2\le\infty$, the uniform strong target index is
$\max\{1,r_1,r_2\}$. The upper bound follows by inclusion and
\eqref{eq:schatten-link}. The rank-one example forces the lower bound
one. With $Z=\C$ and one input rank one, the singular list of the
other input occurs as the output, which forces both input indices
as lower bounds.

\subsection{Oriented contractions on graphs}
\label{subsec:oriented-graphs}

Let a finite directed multigraph have no self-loops. At a vertex $v$,
an operator $T_v$ has one input space $Z_e^-$ for each incoming edge
and one output space $Z_e^+$ for each outgoing edge. Its remaining input
and output factors are external. An edge $e$ carries a connector
$U_e:Z_e^+\to Z_e^-$. The graph contraction $T_{\mathscr D}$ is obtained
by summing the matched indices of the $T_v$ and $U_e$, leaving the
external factors as the domain and range. Fixed reorderings of these
factors are unitary and will be suppressed. A feedback edge set $F$ is
any set of edges whose removal leaves no directed cycle.

\begin{theorem}[Schatten bound for an oriented graph]
\label{thm:oriented-graph}
For finite-dimensional factors and $1\le r\le\infty$, let $q_r$ satisfy
$1/q_r=(1-2/r)_+$, with $q_r=\infty$ when the right side is zero.
For any feedback edge set $F$,
\begin{equation}
\label{eq:oriented-graph-schatten}
 \norm{T_{\mathscr D}}_{\Sch_r}
 \le
 \prod_v \norm{T_v}_{\Sch_r}
 \prod_{e\notin F}\norm{U_e}_{\op}
 \prod_{e\in F}\norm{U_e}_{\Sch_{q_r}} .
\end{equation}
In particular every orientation is bounded with operator-norm connector
costs when $1\le r\le2$, and every acyclic orientation has these costs
for all $r$. The constant one and the index $q_r$ for $r>2$ are optimal
uniformly over the permitted graphs.
\end{theorem}

\begin{proof}
We use a partial trace identity. Let $D:X\otimes Z^-\to Y\otimes Z^+$
and let $\mathcal C_U(D):X\to Y$ close the $Z$ factors with a connector
$U:Z^+\to Z^-$. For every $B:Y\to X$,
\[
 \Tr(B\mathcal C_U(D))=\Tr((B\otimes U)D).
\]
Schatten duality therefore gives the two bounds
\[
 \norm{\mathcal C_U(D)}_{\Sch_1}\le
 \norm D_{\Sch_1}\norm U_{\op},
 \qquad
 \norm{\mathcal C_U(D)}_{\op}\le
 \norm D_{\op}\norm U_{\Sch_1}.
\]
For $r=1$, take $D=\bigotimes_vT_v$ and $U=\bigotimes_eU_e$,
regrouping the internal spaces as above. Multiplicativity of
Schatten norms for tensor products proves \eqref{eq:oriented-graph-schatten}.

For $r=2$, absorb each connector into the output factor of its source
vertex using the ideal property. All edges now identify two copies of
the same Hilbert space. Test the contracted tensor against a unit
Hilbert tensor on the external factors. In orthonormal bases this is a
scalar sum of a product of arrays, with the test tensor as an extra
vertex. Each index being summed occurs at exactly two distinct vertices.
Taking absolute values and applying Cauchy--Schwarz in one index at a
time yields the product of the $\ell^2$ norms of the arrays.
Indeed, the two factors using an index are replaced by their
$\ell^2$ norms in that index, and successive replacements at one
vertex combine into its full $\ell^2$ norm. This is the graph case of
Finner's inequality \cite{Finner}. The test tensor has norm one, and
the other Hilbert tensor norms equal the Hilbert--Schmidt norms of the
vertex operators.

For $r=\infty$, leave the edges of $F$ open and process the resulting
acyclic graph in topological order. Every edge from a processed vertex
to the next vertex enters that vertex. Group these edges into one
Hilbert tensor product. Absorb their tensor-product connector on the
corresponding factors and apply the one-link estimate
\eqref{eq:schatten-link} to the processed and new tensors. If there are
no such edges, take the tensor product. This gives an operator-norm
bound by the vertex and retained connector norms. Close the edges in
$F$ by the partial trace identity with $U=\bigotimes_{e\in F}U_e$ to
obtain the product of their trace norms as the additional factor.

Multilinear complex interpolation \cite{BerghLofstrom} between $r=1$
and $r=2$ preserves the operator norms of all connectors and the
constant one. Between $r=2$ and $r=\infty$, every edge in $F$ interpolates
between the operator norm and trace norm at parameter $1-2/r$, giving
$\Sch_{q_r}$, while all other connector spaces stay fixed. The
interpolation is in the vertex and connector variables of the same
multilinear contraction, so no factor depends on the number of vertices.

One-dimensional factors show sharpness of the constant. For sharpness
of the index, take two vertices joined by two oppositely directed edges,
all four maps equal to $\Id_{\C^n}$, and no external factors. The
contraction is $n$ and the product of vertex norms is $n^{2/r}$.
Take one edge as the feedback set. A uniform bound with
a larger connector index $q$ would require $n\le n^{2/r+1/q}$ for every
$n$, which is impossible for $q>q_r$.
\end{proof}

The same estimate applies to algebraic tensors in arbitrary Hilbert
spaces whenever all connector costs on the right are finite. Restrict
every factor to the finite supports of the vertex tensors and compress
the connectors between them. This leaves the contraction unchanged and
cannot increase the relevant connector norms.

A self-loop may be closed at its vertex first. The same partial trace
identity and the duality between $\Sch_r$ and $\Sch_{r'}$ give
$\norm{\mathcal C_U(D)}_{\Sch_r}\le\norm D_{\Sch_r}\norm U_{\Sch_{r'}}$,
where $1/r+1/r'=1$. A product of such factors treats all self-loops
before applying \eqref{eq:oriented-graph-schatten}. This separate
connector index is sharp for a one-vertex trace by Schatten duality.

\begin{corollary}[Exact exponents for planar split diagrams]
\label{cor:planar-feedback-schatten}
Let a finite closed diagram admit a planar graph after each vertex $v$
is split into $v_{\mathrm{in}}\to v_{\mathrm{out}}$, with its incoming
edges at the former and its outgoing edges at the latter. For indices
$p_v,q_e\in[1,\infty]$, the dimension-free bound
\[
 |T_{\mathscr D}|\le
 \prod_v\norm{T_v}_{\Sch_{p_v}}\prod_e\norm{U_e}_{\Sch_{q_e}}
\]
holds if and only if every directed cycle $C$ satisfies
\[
 \sum_{v\in C}\frac1{p_v}+\sum_{e\in C}\frac1{q_e}\ge1 .
\]
\end{corollary}

\begin{proof}
For necessity, give the wires of $C$ dimension $n$, all others
dimension one, and use identity maps on the resulting factors. The
contraction is $n$, and the right side is $n$ to the power in the
cycle inequality.

In the split graph put $T_v$ on its new edge and use the identity
regroupings at the new endpoints. Each directed cycle corresponds
to one in the original diagram. Write $x_a$ for the reciprocal of
the assigned Schatten index on each split or original edge $a$.
For a feedback edge set $F$, the operator-norm endpoint of
Theorem~\ref{thm:oriented-graph} gives the bound at
$x=\mathbf1_F$. Multilinear complex interpolation and monotonicity
give it throughout the convex hull of these indicators plus the
nonnegative orthant, for $0\le x_a\le1$. By the planar cycle-cover
form of the weighted Lucchesi--Younger theorem
\cite[Corollary~(8.4.39)]{GroetschelLovaszSchrijver}, this dominant
is exactly the region $x\ge0$ and $\sum_{a\in C}x_a\ge1$ for every
directed cycle. This proves sufficiency.
\end{proof}

\begin{proposition}[Mixed indices on bidirected diagrams]
\label{prop:bidirected-feedback}
Let a closed finite loopless diagram have antiparallel edges for each
adjacent pair of vertices. With arbitrary operator-norm connectors,
the dimension-free product bound with vertex factors
$\norm{T_v}_{\Sch_{p_v}}$, $p_v\in[1,\infty]$, holds with constant one exactly when
$1/p_u+1/p_v\ge1$ for every adjacent pair $u,v$.
\end{proposition}

\begin{proof}
Testing a two-cycle on $n$-dimensional identities proves necessity.
Put $a_v=1/p_v$. The vertices of the polytope
$0\le a_v\le1$, $a_u+a_v\ge1$ are half-integral. In fact, in the graph
of tight edges between fractional coordinates, a bipartite component
would allow a small alternating perturbation, while an odd cycle
forces all its coordinates to equal $1/2$.
It thus suffices to prove the bound for $a_v\in\{0,1/2,1\}$.

Absorb the edge connectors by the operator ideal property. Expand the
trace-class maps in singular-vector sums. Each rank-one
term splits its vertex into an input covector and an output vector.
After fixing these terms, no operator-norm vertex has a neighbor
other than a split vertex. A vector or covector can carry several legs,
so keep it as a separate tensor. Each operator-norm vertex together
with the split vectors or covectors on its incident legs can be absorbed
on their associated Hilbert spaces as follows.
Group the tensor product of all output vectors and apply the
tensor product of the operator-norm vertices to its incident factors.
Its Hilbert norm grows by at most the product of their operator norms.
The result is now a collection of Hilbert tensors, with each remaining
vertex map Hilbert--Schmidt. The Hilbert--Schmidt contraction estimate
in Theorem~\ref{thm:oriented-graph} completes the bound.
Sum the singular values at the split vertices and interpolate.
\end{proof}

\subsection{Singular-value transfer and weak ideals}
\label{sec:sharp-weak-schatten-links}

For a compact operator $T$, write
$k_n(T)=\sum_{j=1}^n s_j(T)$ for its $n$th Ky Fan norm.
Write $(a\otimes b)^*$ for the decreasing rearrangement of
$(a_i b_j)_{i,j}$, padded with zeros when necessary.

\begin{theorem}[Singular-value transfer under contraction]
\label{thm:link-singular-values}
For compact $K,L$ of the shapes in \eqref{eq:link-definition},
$\operatorname{Link}_Z(L,K)$ is compact and, for all $n\ge1$,
\begin{equation}
\label{eq:link-singular-values}
 k_n\bigl(\operatorname{Link}_Z(L,K)\bigr)
 \le \frac{73}{64}\sum_{j=1}^n\bigl(s(K)\otimes s(L)\bigr)^*_j .
\end{equation}
For $Z=\C$, the singular-value list of the link equals
$(s(K)\otimes s(L))^*$. More generally, for a finite-dimensional
graph as in Theorem~\ref{thm:oriented-graph}, let $c$ be the
decreasing rearrangement of all products $\prod_v s_{i_v}(T_v)$.
For any feedback edge set $F$ one has
\begin{equation}
\label{eq:graph-ky-fan}
 k_n(T_{\mathscr D})\le \frac{73}{64}
 \left(\prod_{e\notin F}\norm{U_e}_{\op}\right)
 \left(\prod_{e\in F}\norm{U_e}_{\Sch_1}\right)
 \sum_{j=1}^n c_j .
\end{equation}
In an acyclic orientation $F$ may be empty.
\end{theorem}

\begin{proof}
A zero connector gives zero contraction. Otherwise normalize
each connector by its displayed norm. The same proof for arbitrary Hilbert-space links with
finite-rank inputs uses \eqref{eq:schatten-link} at $1$ and $\infty$.
In the singular expansion at each vertex, denote the positive singular
values by $a_{v,i}$ and the corresponding unit rank-one terms by
$E_{v,i}$. For a multi-index $\mathbf i$, let $C_{\mathbf i}$ be the
contraction of the $E_{v,i_v}$ with the fixed connectors and set
$u_{\mathbf i}=\prod_v a_{v,i_v}$. The total contraction is
$\sum_{\mathbf i}u_{\mathbf i}C_{\mathbf i}$. By the trace-norm
endpoint of Theorem~\ref{thm:oriented-graph}, $\norm{C_{\mathbf i}}_{\Sch_1}
\le1$. For $\xi\in\R$, the maps
\[
 T_v(\xi)=\sum_i a_{v,i}^{\mathrm i\xi}E_{v,i}
\]
are partial isometries, so the operator-norm endpoint gives
$\norm{\sum_{\mathbf i} u_{\mathbf i}^{\mathrm i\xi}C_{\mathbf i}}_{\op}
\le1$.

Let $\alpha=1/3$ and set $\varphi(s)=e^s$ for $s\le0$ and
$\varphi(s)=e^{-\alpha s}$ for $s\ge0$. With the Fourier convention
$\varphi(s)=\int_{\R}e^{\mathrm i\xi s}\,\mathrm d\nu(\xi)$,
direct integration gives
\[
 \mathrm d\nu(\xi)=\frac{1+\alpha}
 {2\pi(1-\mathrm i\xi)(\alpha+\mathrm i\xi)}\,\mathrm d\xi .
\]
For $a,b>0$ write
$I(a,b)=\int_0^\infty ((\xi^2+a^2)(\xi^2+b^2))^{-1/2}\,\mathrm d\xi$.
The substitution $x=\xi-ab/\xi$, followed by $x=2y$ and evenness,
gives $I(a,b)=I((a+b)/2,\sqrt{ab})$.
By Cauchy--Schwarz $I(a,b)\le\pi/(2\sqrt{ab})$, hence
\[
 \norm{\nu}_{\mathrm{TV}}
 =\frac{1+\alpha}{\pi}I(1,\alpha)
 \le\sqrt{\frac{1+\alpha}{2\sqrt{\alpha}}}
 =\frac{\sqrt2}{3^{1/4}}=:M .
\]

For $t>0$ put
$R_t=t\sum_{\mathbf i}\varphi(\log(u_{\mathbf i}/t))C_{\mathbf i}$.
It is $t$ times the integral of
$t^{-\mathrm i\xi}\sum_{\mathbf i}u_{\mathbf i}^{\mathrm i\xi}C_{\mathbf i}$
against $\nu$, so $\norm{R_t}_{\op}\le Mt$.
The remainder has trace norm at most
$\sum_{u_{\mathbf i}>t}(u_{\mathbf i}-t^{1+\alpha}u_{\mathbf i}^{-\alpha})$.
With $t=c_n>0$, $S=\sum_{j\le n}c_j$, and $z=nt/S$, the Ky Fan
triangle inequality and Jensen's inequality for $s^{-\alpha}$ give
\[
 k_n(T_{\mathscr D})\le S+Mnt-t^{1+\alpha}\sum_{j\le n}c_j^{-\alpha}
 \le S(1+M z-z^{4/3})\le\tfrac{73}{64}S .
\]
The last inequality uses $\sup_{z\ge0}(M z-z^{4/3})=27M^4/256=9/64$.
If $c_n=0$, the trace-norm triangle inequality gives $k_n(T_{\mathscr D})
\le S$ directly. For compact $K,L$, their finite-rank singular
truncations converge in operator norm, and the links of the truncations
are trace class. Each positive product among the first $n$ is retained
for sufficiently long truncations. Operator-norm continuity of the
link and of $k_n$ gives the estimate and compactness. At $Z=\C$,
the link is the tensor product up to permutation of factors.
\end{proof}

Thus for fully symmetric sequence ideals of compact operators, uniform
bilinear boundedness of a link is equivalent, up to an absolute
constant, to boundedness of the scalar tensor product of the
singular-value lists. Sufficiency follows from
\eqref{eq:link-singular-values} by weak submajorization in the target
ideal. Necessity is the scalar contracted space, with diagonal inputs.

\begin{definition}[Weak and logarithmic Schatten ideals]
\label{def:weak-logarithmic-schatten}
Let $1<r<\infty$, and put $s_n^{**}(T)=k_n(T)/n$. We use
\begin{align}
 \norm T_{r,\infty}^{*}&=\sup_{n\ge1}n^{1/r}s_n(T),\\
 \norm T_{\Sch_{r,\infty}}&=\sup_{n\ge1}n^{1/r}s_n^{**}(T),\\
\label{eq:weak-logarithmic-schatten}
 \norm T_{\Sch_{r,\infty}(\log\Sch)^{-\alpha}}
 &=\sup_{n\ge1} n^{1/r}(\log(e+n))^{-\alpha}s_n^{**}(T)
 \qquad(\alpha\ge0).
\end{align}
The first two expressions are equivalent, since
$\sum_{j\le n}j^{-1/r}\le \frac r{r-1}n^{1-1/r}$. The last two are
fully symmetric complete norms by the Ky Fan triangle inequality
and continuity of each $k_n$ in operator norm. We use the full ideals
of compact operators on which the displayed expressions are finite.
\end{definition}

\begin{theorem}[Sharp uniform weak-Schatten link output]
\label{thm:sharp-weak-schatten-link}
For $1<r<\infty$ and $K,L$ in the full weak Schatten ideals,
\begin{equation}
\label{eq:sharp-weak-schatten-link}
 \norm{\operatorname{Link}_Z(L,K)}_{\Sch_{r,\infty}(\log\Sch)^{-1/r}}
 \le C_r\norm K_{r,\infty}^*\norm L_{r,\infty}^* .
\end{equation}
The logarithmic exponent $1/r$ is optimal even for finite matrices.
\end{theorem}

\begin{proof}
Normalize the two weak quasi-norms to be at most one. For $0<t\le1$
the number of products $s_i(K)s_j(L)\ge t$ is at most
$\sum_{i\le t^{-r}}\lfloor t^{-r}/i\rfloor
 \le C_r t^{-r}\log(e+t^{-1})$.
Consequently the $n$th product in decreasing order is at most
$C_r(\log(e+n)/n)^{1/r}$. Summing through $n$, dividing by $n$ and
using \eqref{eq:link-singular-values} proves \eqref{eq:sharp-weak-schatten-link}.
For sharpness set $Z=\C$ and take both inputs equal to
$\operatorname{diag}(j^{-1/r})_{j\le N}$. Their weak quasi-norms are one.
The link has $m_N=\sum_{j\le N}\lfloor N/j\rfloor\asymp N\log(e+N)$
singular values at least $N^{-1/r}$. Testing the target norm at $m_N$
gives a lower bound of order $(\log(e+N))^{1/r-\alpha}$, which diverges
for $\alpha<1/r$. The logarithmic loss for scalar tensor products is
the discrete counterpart of Astashkin's function-space bound
\cite[Theorem~2.5 and Remark~2.6]{AstashkinTensor}.
\end{proof}

For different weak indices $1<r<s<\infty$, the same argument gives
$\Sch_{r,\infty}\times\Sch_{s,\infty}\to\Sch_{s,\infty}$ with no
logarithmic loss. Indeed, normalized products of size at least $t$
have count at most $t^{-s}\sum_i i^{-s/r}$, and the scalar
contracted space with one rank-one input gives the sharp index $s$.

\subsection{Oriented cuts and Gaussian evaluation}

For a finite operator-valued tensor $K$ of order $m$ with coefficient
operators $K_{\mathbf i}:\cC\to\cE$, its cut at $S\subset[m]$ is the
operator
\begin{equation}
\label{eq:flattening-map}
 \Flat_S(K):\conj{\cH_S}\otimes\cC\longrightarrow\cH_{S^c}\otimes\cE .
\end{equation}
Here $\cH_S$ is the tensor product of the stochastic spaces in $S$ in
increasing order, with $\cH_\varnothing=\C$. The matrix has column
indices from $S$ and $\cC$, and row indices from $S^c$ and $\cE$.
Inner products are linear in the first variable. Put
$\Prof_{m,r}(K)=\max_S\norm{\Flat_S(K)}_{\Sch_r}$.
For $m=0$ use the coefficient operator itself.
Every cut has Hilbert--Schmidt norm equal to the Hilbert tensor norm
of $K$. Permuting stochastic slots permutes the cuts, so their
normalized symmetrization satisfies
\begin{equation}
\label{eq:permutation-profile}
 \Prof_{m,r}(\operatorname{Sym}_mK)\le\Prof_{m,r}(K)
 \qquad(1\le r\le\infty).
\end{equation}

\begin{lemma}[Gaussian evaluation]
\label{thm:gaussian-estimate}
For $1\le p<\infty$, $2\le r<\infty$, and independent real Gaussian
families on the labelled legs,
\begin{equation}
\label{eq:gaussian-estimate}
 \norm{\cT_K^{(m)}}_{L^p(\Omega;\Sch_r)}
 \le C_m(p+r)^{m/2}\Prof_{m,r}(K).
\end{equation}
Here $\cT_K^{(m)}$ is the decoupled evaluation. For stochastic legs
from a common real Hilbert space the same bound holds for $I_m(K)$,
with a constant depending only on $m$.
\end{lemma}

\begin{proof}
We recall the Gaussian integration-by-parts argument for the linear
noncommutative Khintchine bound \cite[Section~3]{vanHandelNotes}.
By self-adjoint dilation consider $X=\sum_jg_j A_j$ with $A_j=A_j^*$.
For $r>2$ set $f(x)=x|x|^{r-2}$. In the spectral basis for $X$,
$D f(X)[A_j]$ has entries $f^{[1]}(\lambda_u,\lambda_v)(A_j)_{uv}$.
This divided-difference formula follows by polynomial approximation of
$f$ and its derivative on compact intervals. The mean value theorem
bounds the divided difference by
$(r-1)(|\lambda_u|^{r-2}+|\lambda_v|^{r-2})$.
Integration by parts in $\E\Tr X f(X)$ therefore gives
\[
 \E\Tr|X|^r\le 2(r-1)\E\Tr\!\left[\left(\sum_j A_j^2\right)|X|^{r-2}\right].
\]
Trace H\"older and Jensen bound the $L^r(\Sch_r)$ norm by
$C\sqrt r\,\norm{(\sum_j A_j^2)^{1/2}}_{\Sch_r}$. At $r=2$ use
orthogonality directly. Undoing the dilation gives $C\sqrt r$ times
the larger Schatten norm of the row and column operators.
The Lipschitz constant is bounded by either norm: substituting
a unit vector in the stochastic slot is an input or output contraction.
The linear Gaussian moment estimate \eqref{eq:partition-linear}
therefore gives the bound with $C\sqrt{p+r}$.
Condition on all but one family. The resulting row and column are
decoupled chaoses of order $m-1$ whose cuts are the cuts of $K$ with
that slot placed on the input and output, respectively.
Iterating in $L^p$ proves \eqref{eq:gaussian-estimate}.

For the same-field bound take independent labelled isonormal
processes $W_j$ and put $W=m^{-1/2}\sum_jW_j$. The Wick exponential
identity, or Gaussian conditional expectation applied to its
coefficients, gives
$\E[\cT_K^{(m)}\mid W]=m^{-m/2}I_m(K)(W)$.
Jensen's inequality gives the claimed bound. This also covers
nonsymmetric $K$ by Wick symmetrization.
\end{proof}

\begin{theorem}[Partition bounds from oriented cuts]
\label{thm:partition-cut-bridge}
Let $K$ be an algebraic operator-valued kernel with stochastic legs
complexified from real Hilbert spaces. For $J\subset[m]$ and
$\pi\in\cP(J)$ put $s=m-|J|$ and define, for $2\le r\le\infty$,
\[
 D_{\pi,r}(K):=
 \max_{T\subset J^c}\ \min_{\mathcal B\subseteq\pi}
 \norm{\Flat_{T\cup(\bigcup_{B\in\mathcal B}B)}(K)}_{\Sch_r}.
\]
Write $a_{\pi,r}(K)$ for the partition seminorm $a_\pi(K)$ with
coefficient norm $\Sch_r$. For $2\le r<\infty$,
\begin{equation}
\label{eq:partition-cut-bridge}
 a_{\pi,r}(K)\le C_s(1+r)^{s/2}D_{\pi,r}(K),\qquad C_0=1.
\end{equation}
For the operator norm, the bound is the infimum of the right side
over $2\le r<\infty$ when $s>0$, and $D_{\pi,\infty}(K)$ when $s=0$.
Summing these bounds with weights $p^{|\pi|/2}$ gives the corresponding
$L^p$ bounds for $\cT_K^{(m)}$ up to an order-dependent constant.
\end{theorem}

\begin{proof}
Fix real unit-ball test tensors $x_B$ for $B\in\pi$ and denote the
partially contracted kernel by $K_x$. For any $T\subset J^c$ and
$\mathcal B\subseteq\pi$, assign each block in $\mathcal B$ to the
input of the original flattening and every other block to its output.
Tensoring with $\bigotimes_{B\in\mathcal B}x_B$ on the input and
pairing against $\bigotimes_{B\in\pi\setminus\mathcal B}x_B$
on the output are Hilbert space contractions. Up to the canonical
permutations and conjugations in \eqref{eq:flattening-map}, composing
the cut of $K$ with these two maps is $\Flat_T(K_x)$. This applies to
arbitrary testing tensors in each block Hilbert tensor product. The
ideal inequality, followed by the minimum over $\mathcal B$ and maximum
over $T$, gives $\Prof_{s,r}(K_x)\le D_{\pi,r}(K)$.
Apply Lemma~\ref{thm:gaussian-estimate} at $p=1$ to $K_x$ and
take the test supremum. For $s=0$, apply the ideal inequality directly.
Monotonicity $\norm{\cdot}_{\op}\le\norm{\cdot}_{\Sch_r}$ gives the
operator-norm estimate, and Theorem~\ref{thm:banach-partition} gives the
last assertion.
\end{proof}

\subsection{Trace-norm compression of nonnegative blocks}
For a block matrix $T=(T_{ij})$, let $C_1(T)=(\norm{T_{ij}}_{\Sch_1})$.
Li and Zhao's proposed inequality (7) for nonnegative tensors
\cite[Section~3]{LiZhaoCompression} specializes to
$\norm T_{\Sch_1}\ge\norm{C_1(T)}_{\Sch_1}$ for nonnegative matrices.

\begin{proposition}[Failure of nonnegative trace-norm compression]
\label{prop:nonnegative-compression}
The trace-norm compression inequality fails for strictly
entrywise-positive, positive semidefinite matrices. Over this class
the infimum of $\norm T_{\Sch_1}/\norm{C_1(T)}_{\Sch_1}$ is zero.
\end{proposition}

\begin{proof}
Let $M$ be the Gram matrix of four real unit vectors with angles
$0,\pi/4,\pi/2,3\pi/4$. Write $A=2\mathbf1\mathbf1^{\mathsf T}$
in dimension four, and form an $8$-by-$8$ matrix $T$ with blocks
\[
 T_{ij}=\frac12
 \begin{pmatrix}
 A_{ij}+M_{ij}&A_{ij}-M_{ij}\\
 A_{ij}-M_{ij}&A_{ij}+M_{ij}
 \end{pmatrix}.
\]
Every entry is at least $1/2$. Applying the two-point orthogonal
Fourier transform on each block and grouping the coordinates gives
$T$ orthogonally equivalent to $A\oplus M$. Both summands are positive
semidefinite, so $\norm T_{\Sch_1}=\Tr A+\Tr M=12$.
The block eigenvalues are $A_{ij}$ and $M_{ij}$, so
$C_1(T)=A+|M|_{\rm ent}$, with entrywise absolute value.
The matrix $|M|_{\rm ent}$ is circulant with first row
$(1,1/\sqrt2,0,1/\sqrt2)$. On $(1,-1,1,-1)^{\mathsf T}$,
$C_1(T)$ therefore has eigenvalue $1-\sqrt2$. As its trace is $12$,
its trace norm is strictly larger than $12$.

In $T^{\otimes m}$ group the outer indices and the block indices
separately. Trace-norm multiplicativity on the blocks gives
$C_1(T^{\otimes m})=C_1(T)^{\otimes m}$. Both positivity properties
are preserved, and the ratio is the $m$th power of the original
ratio, which tends to zero.
\end{proof}

\section{Wick multipliers on groups of polynomial growth}
\label{sec:group-duals}

Oriented cuts give a sharp Schatten convergence criterion for radial
Wick multipliers. Annular concentration yields two-sided almost-sure
singular-value bounds, and local spectral limits determine the leading
constants at the critical radial decay.

\subsection{Group words and cutoffs}
\label{subsec:group-word-model}

Let $\Gamma$ be a finitely generated discrete group with word length
$|\cdot|$ and polynomial volume growth
\begin{equation}
\label{eq:group-growth}
 \#\{x\in\Gamma:|x|\le R\}\le C_\Gamma R^D,
 \qquad R\ge1.
\end{equation}
Let $\lambda_x$ be the left regular shift on $\ell^2(\Gamma)$ and
let $e$ denote the group identity.

Put $\langle x\rangle=1+|x|$ and define
\begin{equation}
\label{eq:group-spectral-operator}
 \Lambda\delta_x=\langle x\rangle\delta_x.
\end{equation}
For dyadic $L\in\Dya:=\{2^j:j=0,1,\ldots\}$, let
\[
 A_L:=\{x\in\Gamma:L\le\langle x\rangle<2L\}
\]
and let $P_L$ be the corresponding coordinate projection.
After enlarging $C_\Gamma$ by a factor depending only on $D$, we use
$\#A_L\le C_\Gamma L^D$ without further comment.

\begin{lemma}[Shift incidence formula]
\label{lem:shift-incidence}
For $w\in\Gamma$,
\begin{equation}
\label{eq:shift-incidence}
 \norm{P_L\lambda_wP_Q}_{\Sch_2}^2
 =\#\{x\in A_Q:wx\in A_L\}
 \le C_\Gamma\min\{L,Q\}^{D}.
\end{equation}
Moreover $\norm{P_L\lambda_wP_Q}_{\cL}\le1$.
\end{lemma}

\begin{proof}
The operator sends $\delta_x$ to $\delta_{wx}$ when
$x\in A_Q$ and $wx\in A_L$, and sends it to zero otherwise.  It is therefore
a partial permutation matrix.  Its Hilbert--Schmidt norm squared is the
number of nonzero columns.  The same number is at most both $\#A_Q$ and
$\#A_L$, which gives \eqref{eq:shift-incidence}.  The operator-norm bound is
immediate.
\end{proof}

For independent standard real Gaussians $g_x$ and a scalar sequence
$c_x$, put $c_x^{(N)}=c_x\mathbf1_{\{|x|\le N\}}$ and define
\begin{equation}
\label{eq:group-mixed-wick-polynomial}
 \Phi_{m,N}(c)=
 \sum_{x_1,\ldots,x_m}
  \wick{g_{x_1}\cdots g_{x_m}}
  \left(\prod_{j=1}^m c_{x_j}^{(N)}\right)\lambda_{x_1\cdots x_m},
 \qquad
 T_{m,N}^{a,b}(c)=\Lambda^{-a}\Phi_{m,N}(c)\Lambda^{-b}.
\end{equation}
The ordered kernel $\widetilde K_{m,N}^{L,Q}(c)$ has coefficient
$(\prod_j c_{x_j}^{(N)})P_L\lambda_{x_1\cdots x_m}P_Q$ at
$(x_1,\ldots,x_m)$. Its normalized symmetrization is denoted by
$K_{m,N}^{L,Q}(c)$. Thus $P_L\Phi_{m,N}(c)P_Q=I_m(K_{m,N}^{L,Q}(c))$.

The volume-asymptotic theorem of Breuillard \cite{Breuillard}, extending
the nilpotent asymptotics of Pansu \cite{Pansu}, shows that the word balls
of every infinite finitely generated group of polynomial growth degree $D$
satisfy
\begin{equation}
\label{eq:exact-word-growth}
 \#B(e,R)=vR^D+o(R^D)
 \qquad(R\to\infty)
\end{equation}
for some $v>0$.

Fix $0<\varepsilon<1/8$ and set
\[
 E_R:=\{q:(1+2\varepsilon)R\le\langle q\rangle
                    <(2-2\varepsilon)R\}.
\]
The exact asymptotic and the triangle inequality give, for all large
dyadic $R$,
\begin{equation}
\label{eq:universal-thick-shell}
 \#E_R\ge cR^D,
 \qquad B(e,\varepsilon R)E_R\subset A_R.
\end{equation}
Indeed, up to the harmless shift in $\langle q\rangle=1+|q|$,
\[
 \#E_R
 =v\bigl((2-2\varepsilon)^D-(1+2\varepsilon)^D\bigr)R^D
  +o(R^D).
\]

\begin{lemma}[Finite-rank Schatten conversion]
\label{lem:finite-rank-schatten-conversion}
Let $P,Q$ be finite-rank orthogonal projections and let $0<r\le2$.
Then every bounded operator $A$ satisfies
\begin{equation}
\label{eq:finite-rank-schatten-conversion}
 \norm{PAQ}_{\Sch_r}
 \le
 \min\{\rank P,\rank Q\}^{1/r-1/2}
 \norm{PAQ}_{\Sch_2}.
\end{equation}
The same inequality holds after taking any $L^p(\Omega)$ norm.
\end{lemma}

\begin{proof}
The rank of $PAQ$ is at most
$k:=\min\{\rank P,\rank Q\}$.  Applying the finite-dimensional
$\ell^r$--$\ell^2$ inequality to its singular values proves
\eqref{eq:finite-rank-schatten-conversion}.  The random assertion follows
pointwise.
\end{proof}

\begin{lemma}[Counting from two Schatten norms]
\label{lem:two-norm-count}
Let $M,V>0$. Suppose $B$ has rank at most $CM$, $\norm B_{\Sch_2}^2\ge cMV^2$,
and $\norm B_{\Sch_4}^4\le CMV^4$. Then at least $c'M$ of its
singular values exceed $c'V$, where $c'>0$ depends only on $c,C$.
\end{lemma}

\begin{proof}
For sufficiently small $\varepsilon>0$, the values below
$\varepsilon V$ account for at most $cMV^2/2$ of the Hilbert--Schmidt
square. If $k$ values remain, their squares sum to at most
$k^{1/2}\norm B_{\Sch_4}^2$ by Cauchy--Schwarz. Thus $k\gtrsim M$.
\end{proof}

\subsection{Sharp all-order transfer on polynomial-growth groups}

At a mixed oriented cut, the fixed stochastic slots appear as insertions
in the ordered group word. The cardinality of the target set controls
such words uniformly in the insertions.

\begin{lemma}[Word incidence and cutoff tails]
\label{lem:lattice-iterated-convolution}
Assume \eqref{eq:group-growth} and let $m\ge2$ and
$(m-1)D/m<\alpha\le D$. Put
\[
 v(x)=\langle x\rangle^{-\alpha},\qquad
 \rho=D-\alpha,\qquad \gamma_k=D-k\rho>0
 \quad(1\le k\le m).
\]
For fixed $a_0,\ldots,a_k\in\Gamma$, write
$W(\mathbf x)=a_0x_1a_1\cdots x_ka_k$. Let $I(E)$ be the sum of
$\prod_jv(x_j)$ over $W(\mathbf x)\in E$, and let $I_{>M}(E)$ restrict
that sum to $\max_j|x_j|>M$. If $\#E\le C_0R^D$ and $M,R\ge1$, then
\begin{align}
\label{eq:lattice-convolution-rearranged}
 I(E)&\lesssim
 \begin{cases}
  R^{k\rho},&\alpha<D,\\
  \log(e+R)^k,&\alpha=D,
 \end{cases}\\
\label{eq:lattice-convolution-tail-rearranged}
 I_{>M}(E)&\lesssim
 \begin{cases}
  R^{k\rho}\min\{1,(R/M)^{\gamma_k}\},&\alpha<D,\\
  \log(e+R)^k\min\{1,(R/M)^\vartheta\},&\alpha=D,\ 0<\vartheta<D.
 \end{cases}
\end{align}
The constants are uniform in the $a_j$. For the product without insertions,
the convolution $h_k=v^{*k}$ satisfies
\begin{align}
\label{eq:lattice-convolution-power}
 h_k(w)&\lesssim\langle w\rangle^{-\gamma_k},
 &&\alpha<D,\\
\label{eq:lattice-convolution-log}
 h_k(w)&\lesssim\langle w\rangle^{-D}\log(e+\langle w\rangle)^{k-1},
 &&\alpha=D.
\end{align}
\end{lemma}

\begin{proof}
Partition each $x_j$ into the dyadic shells $A_{U_j}$ and select one
index $i$ with largest scale $U=U_i$. For each choice of the other
variables, the map $x_i\mapsto W(\mathbf x)$ is a left and right
translation. Thus at most $C\min\{U^D,\#E\}$ values of $x_i$ can occur.
Summing the other $k-1$ scales up to $U$ bounds $I(E)$ by a constant
times
\[
 \sum_{U\in\Dya}
 U^{-\alpha}\min\{U^D,\#E\}
 \left(\sum_{\substack{V\in\Dya\\V\le U}}V^\rho\right)^{k-1}.
\]
For $I_{>M}(E)$ the sum is restricted to $U>M/2$.
For $\rho>0$ the factor in parentheses to the power $k-1$ is
$O(U^{(k-1)\rho})$. Splitting at $U=R$ bounds the full sum by
$CR^{k\rho}$ and the tail for $M\ge R$ by
$CR^DM^{-\gamma_k}$. If $\rho=0$, the factor is
$O(\log(e+U)^{k-1})$. The part $U\le R$ is $O(\log(e+R)^k)$, and
the part $U>R$ is $O(\log(e+R)^{k-1})$. When $M\ge R$ the tail is
bounded by $CR^DM^{-D}\log(e+M)^{k-1}$, which gives
\eqref{eq:lattice-convolution-tail-rearranged} for every $\vartheta<D$.
Finally, if $x_1\cdots x_k=w$, the word-length triangle inequality
implies $U\gtrsim_k\langle w\rangle$. The same dyadic sum with $\#E=1$
and this lower limit proves both pointwise estimates.
\end{proof}

\begin{proposition}[All-order singular profiles]
\label{prop:all-order-singular-lattice-profile}
Assume \eqref{eq:group-growth}, let $m\ge2$, and suppose
\[
 \frac{(m-1)D}{2m}<\sigma\le\frac D2,
 \qquad c_x=\langle x\rangle^{-\sigma}.
\]
Put
\[
 \beta_m=m\left(\frac D2-\sigma\right),\qquad
 \gamma_m=2m\sigma-(m-1)D,
\]
and
\begin{equation}
\label{eq:lattice-singular-scale-weight}
 \Omega_{m,\sigma}(R)=
 \begin{cases}
  R^{\beta_m},&\sigma<D/2,\\
  \log(e+R)^{m/2},&\sigma=D/2.
 \end{cases}
\end{equation}
For $R_+=\max\{L,Q\}$ and $R_-=\min\{L,Q\}$,
\begin{equation}
\label{eq:all-order-singular-lattice-profile}
 \Prof_{m,r}(K_{m,N}^{L,Q}(c))
 \le C\Omega_{m,\sigma}(R_+)R_-^{D/r},
 \qquad2\le r\le\infty.
\end{equation}
If $N\ge M$ and $\sigma<D/2$, then
\begin{align}
\label{eq:all-order-singular-lattice-tail}
 &\Prof_{m,r}(K_{m,N}^{L,Q}(c)-K_{m,M}^{L,Q}(c))\notag\\
 &\quad\le C\Omega_{m,\sigma}(R_+)R_-^{D/r}
 \min\{1,(R_+/M)^{\gamma_m/2}\}.
\end{align}
At $\sigma=D/2$, for every $0<\vartheta<D$ the right side is replaced by
\begin{equation}
\label{eq:all-order-singular-lattice-tail-critical}
 C_\vartheta\Omega_{m,D/2}(R_+)R_-^{D/r}
 \min\{1,(R_+/M)^{\vartheta/2}\}.
\end{equation}
\end{proposition}

\begin{proof}
Write $u(x)=\langle x\rangle^{-\sigma}$, $v=u^2$, and fix
$S\subset[m]$, $k=\#S$. In the input basis
$(\mathbf x_S,q)$ and output basis $(\mathbf x_{S^c},l)$, the absolute
matrix entries of the ordered flattening are
\[
 \left(\prod_{j=1}^mu_N(x_j)\right)
 \mathbf1_{\{l=x_1\cdots x_mq\}}.
\]
Use the Schur weights
\[
 p_{\rm in}=\prod_{j\in S}u(x_j),
 \qquad p_{\rm out}=\prod_{j\notin S}u(x_j).
\]
For a fixed row the output variables and $l$ are fixed. Given
$\mathbf x_S$, there is at most one possible $q$, and it lies in $A_Q$
exactly when $x_1\cdots x_m\in l A_Q^{-1}$. Thus the weighted row sum
divided by $p_{\rm out}$ is bounded by $I(l A_Q^{-1})$ for a word with
$k$ free variables. In a fixed column the corresponding target is
$A_L q^{-1}$ and there are $m-k$ free variables. The cardinalities are at
most $CQ^D$ and $CL^D$. Lemma~\ref{lem:lattice-iterated-convolution}
gives the row and column constants, with bound one for an empty set of
free variables,
\[
 A_S\lesssim
 \begin{cases}Q^{k(D-2\sigma)},&\sigma<D/2,\\
 \log(e+Q)^k,&\sigma=D/2,
 \end{cases}
 \quad
 B_S\lesssim
 \begin{cases}L^{(m-k)(D-2\sigma)},&\sigma<D/2,\\
 \log(e+L)^{m-k},&\sigma=D/2.
 \end{cases}
\]
The two-weight Schur test therefore gives the cut-specific endpoint
\begin{equation}
\label{eq:lattice-cut-specific-endpoint}
 \norm{\Flat_S\widetilde K_{m,N}^{L,Q}}_{\cL}
 \lesssim
 \begin{cases}
 Q^{k(D-2\sigma)/2}L^{(m-k)(D-2\sigma)/2},&\sigma<D/2,\\
 \log(e+Q)^{k/2}\log(e+L)^{(m-k)/2},&\sigma=D/2.
 \end{cases}
\end{equation}
Thus every one of the $2^m$ cuts is bounded by
$C\Omega_{m,\sigma}(R_+)$.  No identification of different cuts is used.

At the Hilbert endpoint every cut has the same norm, and
\[
 \norm{\Flat_S\widetilde K_{m,N}^{L,Q}}_{\Sch_2}^2
 \le\sum_{q\in A_Q}\sum_{l\in A_L}h_m(lq^{-1}).
\]
Summing first in either shell and using
\eqref{eq:lattice-convolution-rearranged} yields
\[
 \norm{\Flat_S\widetilde K_{m,N}^{L,Q}}_{\Sch_2}
 \lesssim\Omega_{m,\sigma}(R_+)R_-^{D/2}.
\]
Interpolation with \eqref{eq:lattice-cut-specific-endpoint} proves
\eqref{eq:all-order-singular-lattice-profile} for the ordered kernel.

For the cutoff difference, use
\[
 \left|\prod_ju_N(x_j)-\prod_ju_M(x_j)\right|^2
 \le\sum_{h=1}^m\prod_jv(x_j)\mathbf1_{\{|x_h|>M\}}.
\]
For the operator endpoint, the corresponding unsquared difference is
bounded by the sum of the $m$ nonnegative kernels obtained by imposing
$|x_h|>M$.  Thus the same position-by-position decomposition applies to
the Schur test.
If $h\in S$, so $k\ge1$, the row constant gains
$\min\{1,(Q/M)^{\gamma_k}\}$; the Schur square root therefore gives the
operator norm the factor
$\min\{1,(Q/M)^{\gamma_k/2}\}$.  If $h\notin S$, so $m-k\ge1$, the
column constant similarly gives
$\min\{1,(L/M)^{\gamma_{m-k}/2}\}$.  For the relevant nonempty side
$j\in\{k,m-k\}$ one has $\gamma_j\ge\gamma_m$, and $L,Q\le R_+$, so
both factors are bounded by
$\min\{1,(R_+/M)^{\gamma_m/2}\}$. At the Hilbert endpoint, the word-incidence
tail bound for $m$ variables gives exponent $\gamma_m$, and taking the square root
again yields $\gamma_m/2$.  At $\sigma=D/2$, the same argument with
$0<\vartheta<D$ gives $\vartheta/2$.  Interpolation and
\eqref{eq:permutation-profile} finish the proof.
\end{proof}

\begin{proposition}[Critical cutoff tail]
\label{prop:critical-cutoff-sharp}
Under the assumptions of Proposition~\ref{prop:all-order-singular-lattice-profile},
let $\sigma=D/2$, $R_+=\max\{L,Q\}$, $R_-=\min\{L,Q\}$, and write
$\Delta K=K_{m,N}^{L,Q}(c)-K_{m,M}^{L,Q}(c)$. If $N\ge M\ge C_m R_+$, then
\begin{equation}\label{eq:critical-cutoff-sharp}
 \Prof_{m,r}(\Delta K)
 \lesssim R_-^{D/r}(R_+/M)^{D/2}\log(e+M)^{(m-2)/2},
 \qquad 2\le r\le\infty.
\end{equation}
If \eqref{eq:exact-word-growth} holds, the reverse bound follows from the
empty cut for $Q=1$, all sufficiently large $L=R$, and $N\ge 5M$.
\end{proposition}

\begin{proof}
In each nonzero term, $x_1\cdots x_m=lq^{-1}$ has length at most $4R_+$.
If the largest $\langle x_j\rangle$ belongs to the dyadic shell at $U\gtrsim M$,
then a second slot has length $\gtrsim_m U$. Indeed, solving for the largest
slot bounds its length by $4R_+$ plus the sum of the other lengths. Split
the ordered cutoff difference by $U$ and the positions of these two slots.
For the Schur weights from Proposition~\ref{prop:all-order-singular-lattice-profile},
the remaining weights on a free side are $v(x)=\langle x\rangle^{-D}$.
If that side has $k$ free slots and $j\in\{0,1,2\}$ of the two large slots,
its Schur constant is bounded by
\[
 \begin{cases}
  C\log(e+U)^k,&j=0,\\
  C\,V^D U^{-D}\log(e+U)^{k-j},&j\ge1,
 \end{cases}
\]
where $V=Q$ for rows and $V=L$ for columns. In the second case, fixing the
word value and the other variables determines one large slot. Its weight
is $O(U^{-D})$, and the other large slot, if present, has bounded annular
mass. Multiplying the row and column bounds gives at most
$C(R_+/U)^D\log(e+U)^{m-2}$ since $U\gtrsim R_+$. For the squared
Hilbert--Schmidt bound fix a spatial index on the smaller side and use
the same count, now with all $m$ slots free, then sum over the smaller
shell. This adds at most the factor $CR_-^D$. Interpolating the operator
and Hilbert--Schmidt bounds and summing over $U$ proves \eqref{eq:critical-cutoff-sharp}, with
symmetrization contractive.

For the lower bound fix $l\in A_R$ and take $2M\le|x_1|\le3M$,
$|x_j|\le M/(10m)$ for $2\le j\le m-1$, and
$x_m=(x_1\cdots x_{m-1})^{-1}l$. Then all variables lie in the cutoff
at $N$ and $|x_1|>M$. By \eqref{eq:exact-word-growth} the $x_1$ annulus
has $\gtrsim M^D$ points, the freely summed small slots each contribute
$\gtrsim\log(e+M)$, and the weights of $x_1,x_m$ are $\gtrsim M^{-D}$.
The squared ordered coefficient mass per $l$ is therefore
$\gtrsim M^{-D}\log(e+M)^{m-2}$. For a nonnegative coefficient tensor
$f_l$, $\|\operatorname{Sym} f_l\|^2\ge\|f_l\|^2/m!$ by the permutation
average. The empty cut has one spatial input, so its only singular value
is the square root of this mass summed over $A_R$.
\end{proof}

\begin{theorem}[Sharp all-order singular transfer on polynomial-growth groups]
\label{thm:all-order-singular-lattice-transfer}
Let $\Gamma$ be an infinite finitely generated group of polynomial
growth degree $D$. Fix $m\ge2$, $a,b\ge0$, $0<\sigma\le D/2$, and
$c_x=\langle x\rangle^{-\sigma}$. If
\begin{equation}
\label{eq:all-order-local-threshold}
 \sigma\le\frac{(m-1)D}{2m},
\end{equation}
then no $a,b\ge0$ make $(T_{m,N}^{a,b}(c))_N$ Cauchy in
$L^2(\Omega;\Sch_r)$ for any $0<r<\infty$.

Suppose instead that $(m-1)D/(2m)<\sigma\le D/2$, and define
\[
 \beta_m:=m\left(\frac D2-\sigma\right),\qquad
 \gamma_m:=2m\sigma-(m-1)D.
\]
For
$0<r<\infty$, the cutoffs are Cauchy in
$L^2(\Omega;\Sch_r)$ if and only if
\begin{equation}
\label{eq:all-order-singular-lattice-range}
 a>\beta_m,\qquad b>\beta_m,
 \qquad a+b>\beta_m+\frac Dr.
\end{equation}
Whenever these conditions hold, the cutoffs converge in every finite
$L^p(\Omega;\Sch_r)$ and almost surely in $\Sch_r$, using the usual
quasi-norm when $r<1$. Denote their common limit by $T_m^{a,b}(c)$.

Set
\begin{equation}
\label{eq:all-order-singular-lattice-margin}
 \eta_*:=\min\left\{
  \frac{\gamma_m}{2},\ a-\beta_m,\ b-\beta_m,
  a+b-\beta_m-\frac Dr
 \right\}.
\end{equation}
When \eqref{eq:all-order-singular-lattice-range} holds, every
$0<\eta<\eta_*$ and $1\le p<\infty$ satisfy
\begin{equation}
\label{eq:all-order-singular-lattice-lp-rate}
 \norm{T_{m,N}^{a,b}(c)-T_m^{a,b}(c)}_{L^p(\Omega;\Sch_r)}
 \le C_\eta(p+r)^{m/2}N^{-\eta}.
\end{equation}
For a common realization and $t\ge1$,
\begin{equation}
\label{eq:all-order-singular-lattice-probability-rate}
 \Prob\!\left(
  \norm{T_{m,N}^{a,b}(c)-T_m^{a,b}(c)}_{\Sch_r}
  >\mathrm e C_\eta(t+r)^{m/2}N^{-\eta}
 \right)\le\mathrm e^{-t},
\end{equation}
and
\begin{equation}
\label{eq:all-order-singular-lattice-as-rate}
 \norm{T_{m,N}^{a,b}(c)-T_m^{a,b}(c)}_{\Sch_r}
 =O_\omega(N^{-\eta})
 \qquad\text{almost surely}.
\end{equation}
At $\sigma=D/2$, one has $\gamma_m=D$, so
\[
 \eta_*=\min\left\{\frac D2,a,b,a+b-\frac Dr\right\}.
\]
The constants in the rate estimates may depend on $\eta$ and on all fixed
parameters in the theorem. The sufficient conditions and rates use only
the volume upper bound \eqref{eq:group-growth} with the displayed $D$.
\end{theorem}

\begin{proof}
Assume first that the local threshold is satisfied strictly.  The
same-field Wick estimate and
Proposition~\ref{prop:all-order-singular-lattice-profile} give the block
tail bound with moment factor $(p+r)^{m/2}$ when $2\le r<\infty$.
For $0<r<2$, apply that estimate at $r=2$ and then
Lemma~\ref{lem:finite-rank-schatten-conversion}.  The spatial block has
rank at most $CR_-^D$, so the resulting bound has exactly the factor
$R_-^{D/r}$, the same ultraviolet tail factor, and moment factor
$(p+2)^{m/2}\lesssim_m(p+r)^{m/2}$.  Thus the same block tail bound holds
throughout $0<r<\infty$.  If $\sigma<D/2$, use
\[
 \min\{1,(R_+/M)^{\gamma_m/2}\}
 \le M^{-\eta}R_+^\eta
 \qquad(0<\eta<\gamma_m/2).
\]
At $\sigma=D/2$, first choose $2\eta<\vartheta<D$ and use
\eqref{eq:all-order-singular-lattice-tail-critical}.  In the region
$L\ge Q$, after inserting the smoothing factors the dyadic majorant is
\[
 M^{-\eta}L^{-a+\beta_m+\eta}Q^{-b+D/r},
\]
with the additional factor $\log(e+L)^{m/2}$ when $\sigma=D/2$.
The inner dyadic sum is
\begin{equation}
\label{eq:singular-dyadic-inner-sum}
 \sum_{Q\le L}Q^{-b+D/r}
 \lesssim
 \begin{cases}
  1,&b>D/r,\\
  \log(e+L),&b=D/r,\\
  L^{D/r-b},&b<D/r.
 \end{cases}
\end{equation}
For $r\ge1$, use the triangle inequality in $L^p(\Omega;\Sch_r)$.
For $r<1$, Rotfel'd's inequality
$\norm{A+B}_{\Sch_r}^r\le\norm A_{\Sch_r}^r+\norm B_{\Sch_r}^r$
\cite{SimonTraceIdeals}, followed by Minkowski in $L^{p/r}$, bounds
the $r$th power of the sum by the sum of the $r$th powers of the block
majorants. The dyadic powers have the same summability conditions in
both cases. Summing over $L\ge Q$ and $Q\ge L$ requires
\[
 \eta<a-\beta_m,\qquad \eta<b-\beta_m,
 \qquad \eta<a+b-\beta_m-D/r .
\]
These strict margins absorb the endpoint logarithms.
Completeness identifies the resulting block sum with the cutoff
difference by its matrix entries.
This proves the uniform $L^p$ Cauchy estimate and
\eqref{eq:all-order-singular-lattice-lp-rate}.  Markov's inequality with
$p=t$ proves \eqref{eq:all-order-singular-lattice-probability-rate}; use a
slightly larger rate exponent and $t=2\log N$ to obtain
\eqref{eq:all-order-singular-lattice-as-rate} by Borel--Cantelli.

We prove necessity. For $w\in\Gamma$, let
\[
 f_w(\mathbf x)=\mathbf1_{\{x_1\cdots x_m=w\}}
                         \prod_{j=1}^m\langle x_j\rangle^{-\sigma},
 \qquad H_m(w)=\norm{f_w}_{\ell^2(\Gamma^m)}^2.
\]
Lemma~\ref{lem:lattice-iterated-convolution} gives $H_m(w)<\infty$.
The cutoff coefficients $Z_{w,N}$ therefore converge in $L^2(\Omega)$
to $Z_w=I_m(\operatorname{Sym}f_w)$. The Wiener--It\^o isometry and
positivity give
\begin{equation}
\label{eq:lattice-coefficient-variance}
 \E|Z_w|^2=m!\norm{\operatorname{Sym}f_w}_{\ell^2}^2
 =\sum_{\pi\in\mathfrak S_m}\ip{f_w}{f_w\circ\pi}\ge H_m(w).
\end{equation}
The identity permutation supplies the last inequality. The same formula
holds for nonnegative cutoff differences.
Define the canonical array $T$ by the entrywise limits
\begin{equation}
\label{eq:lattice-limit-entries}
 \ip{T\delta_q}{\delta_l}
 =\langle l\rangle^{-a}\langle q\rangle^{-b}Z_{lq^{-1}}.
\end{equation}
On $\mathbb Z^D$, all summands in \eqref{eq:lattice-coefficient-variance}
are equal by symmetry, so $\E|Z_w|^2=m!H_m(w)$.

Choose the thick shell $E_R\subset A_R$ in
\eqref{eq:universal-thick-shell}, and write
$v(x)=\langle x\rangle^{-2\sigma}$. Restricting the first $m-1$ variables
to $B(e,\varepsilon R/(m-1))$ and the last one to $E_R$ gives the product
bound
\[
 \sum_{l\in A_R}H_m(l)
 \ge
 \left(
  \sum_{x\in B(e,\varepsilon R/(m-1))}v(x)
 \right)^{m-1}
 \sum_{y\in E_R}v(y).
\]
Consequently,
\begin{equation}
\label{eq:lattice-one-sided-lower}
 \sum_{l\in A_R}H_m(l)\gtrsim
 \begin{cases}
  R^{2\beta_m},&\sigma<D/2,\\
  \log(e+R)^{m-1},&\sigma=D/2.
 \end{cases}
\end{equation}
Thus
\[
 \norm{P_RT\delta_e}_{L^2(\Omega;\ell^2)}^2
 \gtrsim R^{-2a}\sum_{l\in A_R}H_m(l).
\]
If the cutoffs converge in $L^2(\Omega;\Sch_r)$, their limit must
represent this array. Since $P_R\to0$ strongly and
$\norm{P_RT\delta_e}_{\ell^2}\le\norm T_{\Sch_r}$, dominated convergence
shows that the left side tends to zero.  This forces $a>\beta_m$ below the
endpoint and $a>0$ at the endpoint. The analogous row estimate for
$T^*$ uses $H_m(w^{-1})=H_m(w)$, which follows by reversing and inverting
the tuple. It forces the condition on $b$.

For the joint condition, restrict $q$ to $E_R$, put $l=wq$, and
restrict all variables in $H_m(w)$ to $B(e,\varepsilon R/m)$. Summing
first over these variables gives
\[
 \sum_{q,l\in A_R}H_m(lq^{-1})
 \ge \#E_R
 \left(
  \sum_{x\in B(e,\varepsilon R/m)}v(x)
 \right)^m.
\]
Hence
\begin{equation}
\label{eq:lattice-two-sided-lower}
 \sum_{q,l\in A_R}H_m(lq^{-1})\gtrsim
 \begin{cases}
  R^{D+2\beta_m},&\sigma<D/2,\\
  R^D\log(e+R)^m,&\sigma=D/2.
 \end{cases}
\end{equation}
Let $B_R=P_RTP_R$.  Equations \eqref{eq:lattice-coefficient-variance},
\eqref{eq:lattice-limit-entries}, and
\eqref{eq:lattice-two-sided-lower} imply
\begin{equation}
\label{eq:lattice-block-hilbert-lower}
 \norm{B_R}_{L^2(\Omega;\Sch_2)}
 \gtrsim R^{-a-b+D/2}\Omega_{m,\sigma}(R).
\end{equation}
Hilbert-valued hypercontractivity and Paley--Zygmund applied to
the squared Hilbert--Schmidt norm show that $\norm{B_R}_{\Sch_2}$
is at least $cR^{D/2}V_R$ with probability bounded below uniformly,
where $V_R=R^{-a-b}\Omega_{m,\sigma}(R)$. The block estimate at $q=4$
and Markov's inequality give $\norm{B_R}_{\Sch_4}\le C R^{D/4}V_R$
except on an event of arbitrarily small probability if $C$ is large.
The block estimate passes to the entrywise limit on each fixed
finite-dimensional shell. On the intersection, which has probability
bounded below, Lemma~\ref{lem:two-norm-count} gives at least $c'R^D$
singular values of size at least $c'V_R$. Consequently for every $r>0$,
\[
 \norm{B_R}_{L^2(\Omega;\Sch_r)}\gtrsim R^{D/r}V_R.
\]
If $T\in L^2(\Omega;\Sch_r)$, compact approximation and domination
by $\norm T_{\Sch_r}$ make the left side tend to zero. This forces
$a+b>\beta_m+D/r$, including the strict critical inequality.

Finally suppose \eqref{eq:all-order-local-threshold} holds. At each
large dyadic scale $R$, take $x_1\in E_R$, choose
$x_2,\ldots,x_{m-1}\in B(e,R)$ and set
$x_m=(x_1\cdots x_{m-1})^{-1}$. All lengths are $O_m(R)$ and there are
$\gtrsim R^{(m-1)D}$ such tuples. The families are disjoint across scales
by their first coordinate. Their contribution at scale $R$ to $H_m(e)$
is at least $cR^{(m-1)D-2m\sigma}$, so $H_m(e)=\infty$.
Let $H_{m,N}(e)$ denote the same sum with $\max_j|x_j|\le N$.
The positivity argument for the finite identity coefficient gives
\[
 \E\left|\ip{T_{m,N}^{a,b}(c)\delta_e}{\delta_e}\right|^2
 \ge H_{m,N}(e)\uparrow\infty .
\]
Thus the identity coefficient cannot be $L^2$ Cauchy for any $a,b$.
\end{proof}

For example, the integer Heisenberg group $H_3(\mathbb Z)$ has $D=4$.
For cubic Wick powers, the local threshold is $\sigma>4/3$ in the
singular range $\sigma\le2$, and the Schatten conditions are
$a,b>6-3\sigma$ and $a+b>6-3\sigma+4/r$.

\subsection{Annular concentration and critical ideals}
\label{subsec:annular-concentration}

\begin{lemma}[Annular concentration]
\label{lem:group-annular-concentration}
Assume the strict local threshold and the polynomial-growth group
hypotheses of Theorem~\ref{thm:all-order-singular-lattice-transfer}. Let
$A^{L,Q}$ be the entrywise limit of $A_N^{L,Q}:=P_L\Phi_{m,N}(c)P_Q$.
For each even integer $q\ge2$, almost surely there is a dyadic threshold
beyond which, uniformly in dyadic $L,Q$ with $R_+:=\max\{L,Q\}$ above
that threshold and $R_-:=\min\{L,Q\}$,
\begin{equation}\label{eq:group-annular-as-upper}
 \norm{A^{L,Q}}_{\Sch_q}\le C_q R_-^{D/q}\Omega_{m,\sigma}(R_+).
\end{equation}
For sufficiently large dyadic $R$, almost surely,
\begin{equation}\label{eq:group-annular-as-lower}
 \norm{A^{R,R}}_{\Sch_2}\ge c R^{D/2}\Omega_{m,\sigma}(R).
\end{equation}
The constants are deterministic. The corresponding bounds for a single
pair hold with probability tending to one uniformly as $R_+\to\infty$.
\end{lemma}

\begin{proof}
Write $q=2k$, $R=R_+$, $M=R_-$, $\ell(R)=\log(e+R)$, and
$\Omega(R)=\Omega_{m,\sigma}(R)$. The block profile estimates and
same-field evaluation give
\begin{align*}
 \norm{A_N^{L,Q}}_{L^p(\Omega;\Sch_q)}
 &\le C_{p,q} M^{D/q}\Omega(R),\\
 \norm{A^{L,Q}-A_N^{L,Q}}_{L^p(\Omega;\Sch_q)}
 &\le C_{p,q} M^{D/q}\Omega(R)(R/N)^\delta,\qquad N\ge R,
\end{align*}
for some $\delta>0$ independent of the scales. In particular the
limits in finite block spaces exist. Take
$N=\lceil R\ell(R)^B\rceil$ with $B\delta\ge1$.

Let $Y=\norm{A_N^{L,Q}}_{\Sch_{2k}}^{2k}$. Compute the trace on the
smaller shell using either $(A_N^* A_N)^k$ or $(A_N A_N^*)^k$.
We claim
\begin{equation}\label{eq:group-annular-trace-variance}
 \operatorname{Var}(Y)
 \le C M^{2D}
 \left(\sum_{|x|\le N}\langle x\rangle^{-2\sigma}\right)^{2km-1}.
\end{equation}
In the expansion of $\E Y^2$, each of the two traces contains $2k$
real Wick factors of order $m$. Adjoint factors invert the shift word
and leave the Wick factors and positive coefficients unchanged.
The Wick pairing formula excludes pairs
inside a single factor. Pairings having no pair between the two
traces give exactly $(\E Y)^2$. Fix a remaining pairing. Every pair
carries one group variable and the weight $\langle x\rangle^{-2\sigma}$.
The trace of the associated product of shifts and projections counts
its starting points on the chosen shell, with a point counted
only if the ordered product of the shifts is the identity.
A variable from a pair joining the two traces occurs exactly once in
the first trace word, possibly inverted in an adjoint factor.
With the other pair variables fixed the identity condition determines
this variable uniquely. Its weight is at most one. Each trace has at
most $CM^D$ starting points. Dropping the remaining restrictions
and summing the other $2km-1$ weights proves
\eqref{eq:group-annular-trace-variance}.

Set $\rho=D-2\sigma$. When $\rho>0$, the weighted ball sum is $O(N^\rho)$.
Dividing \eqref{eq:group-annular-trace-variance} by
$M^{2D}\Omega(R)^{4k}$ gives
$O(R^{-\rho}\ell(R)^{B\rho(2km-1)})$. When $\rho=0$, the sum is
$O(\log(e+N))=O(\ell(R))$, and the same variance ratio is
$O(\ell(R)^{-1})$. The degree of $Y-\E Y$ is at most $2km$.
Scalar Gaussian hypercontractivity therefore bounds its $L^p$ norm
by $C_{p,m,k} M^D\Omega(R)^{2k}\ell(R)^{-1/2}$ in both cases.
The probability that the centered trace exceeds
$M^D\Omega(R)^{2k}\ell(R)^{-1/4}$ in absolute value is
$O_p(\ell(R)^{-p/4})$. The $\Sch_q$ norm of the cutoff difference stays below
$M^{D/q}\Omega(R)\ell(R)^{-1/4}$ outside an event of probability
$O_p(\ell(R)^{-3p/4})$. There are $O(j)$ dyadic pairs with $R=2^j$.
Taking $p>8$ and using Borel--Cantelli makes both bounds eventual
and uniform. Since $\E Y\le C_q M^D\Omega(R)^{2k}$, this proves
\eqref{eq:group-annular-as-upper}.

For the diagonal at $k=1$, apply
\eqref{eq:lattice-block-hilbert-lower} with any fixed $a,b$ in the
Hilbert--Schmidt convergence range. Removing the weights gives
$\norm{A^{R,R}}_{L^2(\Omega;\Sch_2)}\gtrsim R^{D/2}\Omega(R)$.
The cutoff error is of smaller order
in this norm, so $\E Y\gtrsim R^D\Omega(R)^2$ at the selected cutoff.
The centered trace bound and the cutoff difference now give
\eqref{eq:group-annular-as-lower}. The same probability bounds before
Borel--Cantelli give the claimed convergence in probability of the
annular estimates.
\end{proof}

\begin{corollary}[Sharp approximation numbers and critical ideals]
\label{cor:group-best-rank}
In the strict range $(m-1)D/(2m)<\sigma\le D/2$ of
Theorem~\ref{thm:all-order-singular-lattice-transfer}, suppose $a,b>\beta_m$.
Write $T=T_m^{a,b}(c)$ and
$u_n=n^{-(a+b)/D}\Omega_{m,\sigma}(n^{1/D})$.
For every $n\ge1$ and $1\le p<\infty$,
\begin{equation}\label{eq:group-best-rank}
 c_0 u_n\le \norm{\mu_n(T)}_{L^p(\Omega)}
 \le C_0(p+1)^{m/2}u_n.
\end{equation}
Almost surely,
\begin{equation}\label{eq:group-best-rank-pathwise-lower}
 0<c_1\le \liminf_{n\to\infty}\frac{\mu_n(T)}{u_n}
 \le \limsup_{n\to\infty}\frac{\mu_n(T)}{u_n}\le C_1<\infty.
\end{equation}
The constants are deterministic. For fixed $\Gamma,m,\sigma$, the
pathwise bounds hold on one event for all $a,b>\beta_m$.
At $r_*=D/(a+b-\beta_m)$, one has
\begin{equation}\label{eq:group-critical-log-weak}
 \sup_{n\ge1} n^{1/r_*}\log(e+n)^{-\theta}\mu_n(T)<\infty
 \quad\text{almost surely}\quad\Longleftrightarrow\quad
 \theta\ge
 \begin{cases}
   0,&\sigma<D/2,\\
   m/2,&\sigma=D/2 .
 \end{cases}
\end{equation}
Thus below $\sigma=D/2$ the critical weak Schatten ideal contains $T$,
and at $\sigma=D/2$ the sharp logarithmic enlargement is required.
\end{corollary}

\begin{proof}
Choose an even $q$ with $D/q<\min\{a,b\}$. Write
$\Pi_R=\sum_{L<R}P_L$ and $Q_R=I-\Pi_R$. The block moment bound
gives
\[
 \norm{P_LT P_Q}_{L^p(\Omega;\Sch_q)}
 \lesssim_q (p+1)^{m/2}L^{-a}Q^{-b}
 \Omega_{m,\sigma}(R_+)R_-^{D/q}.
\]
Lemma~\ref{lem:group-annular-concentration} gives the same bound
pointwise without the $p$ factor for all sufficiently large $R_+$.
The annular event is independent of the smoothing weights. Intersect it
over the even integers $q\ge2$ to obtain the stated simultaneous bounds.
In the sums over $L\ge Q\ge R$, summing $Q^{-b+D/q}$ costs
$O(R^{-b+D/q})$ and summing $L^{-a}\Omega_{m,\sigma}(L)$ costs
$O(R^{-a}\Omega_{m,\sigma}(R))$. The other orientation uses
$a>D/q$ and $b>\beta_m$. Thus $Q_RTQ_R$ has $\Sch_q$ norm at most
$C R^{-a-b+D/q}\Omega_{m,\sigma}(R)$ eventually almost surely, and
the same estimate times $(p+1)^{m/2}$ holds in $L^p$ for every $R$.
Since $T-Q_RTQ_R=\Pi_RT+Q_RT\Pi_R$ has rank at most $CR^D$,
\[
 \mu_n(T)\le R^{-D/q}\norm{Q_RTQ_R}_{\Sch_q}
 \quad\text{if}\quad n\ge 2\operatorname{rank}\Pi_R+\lceil R^D\rceil .
\]
Taking dyadic $R$ comparable to a sufficiently small constant times
$n^{1/D}$ proves both upper bounds. The moment bound at bounded
indices follows from that of $T$ in $\Sch_q$.

For the lower bound put $V_R=R^{-a-b}\Omega_{m,\sigma}(R)$ and
$B_R=P_RTP_R$. The diagonal weight operators have singular values
comparable to $R^{-a}$ and $R^{-b}$ on $A_R$.
The annular bounds at $q=2,4$ therefore give
$\norm{B_R}_{\Sch_2}^2\ge c R^D V_R^2$ and
$\norm{B_R}_{\Sch_4}^4\le C R^D V_R^4$ eventually almost surely,
and with probability tending to one for each $R$.
Lemma~\ref{lem:two-norm-count} gives at least $c'R^D$ singular
values above $c'V_R$. Compression gives $\mu_j(B_R)\le\mu_j(T)$.
Taking dyadic $R$ with $c'R^D\ge n$ and $R^D\asymp n$ gives the
pathwise lower bound and, by the uniform positive probabilities for
large $R$, the moment lower bound. Monotonicity extends the latter
to bounded indices. Finally \eqref{eq:group-critical-log-weak}
is a direct consequence of the two-sided decay.
\end{proof}

\begin{corollary}[Almost-sure Schatten membership]
\label{cor:group-wick-pathwise-schatten}
Above the strict local threshold, let $T$ be the array of entrywise
limits in \eqref{eq:lattice-limit-entries}. It represents a compact
operator almost surely if $a,b>\beta_m$, and otherwise almost surely
does not represent a bounded operator. For every $0<r<\infty$, it
represents a $\Sch_r$ operator almost surely exactly when
\[
 a>\beta_m,\qquad b>\beta_m,\qquad a+b>\beta_m+D/r.
\]
For \(0<\sigma\le D/2\), the cutoffs converge in probability in
$\Sch_r$ exactly when the strict local threshold and these inequalities
hold. At or below the local threshold the identity coefficients of
the cutoffs are unbounded almost surely.
\end{corollary}

\begin{proof}
In the compact range, \eqref{eq:group-best-rank-pathwise-lower}
gives the Schatten criterion by summing the $r$th powers.
For the failure of a one-sided condition, use the zero--one law
for a Borel linear subspace of matrix arrays. Indeed, conditional
on all but the first $n$ Gaussian coordinates, each entry is a
polynomial in those $n$ coordinates of degree at most $m$. Modulo
a linear subspace, either this polynomial array vanishes identically,
or a scalar linear functional on its finite-dimensional coefficient
span gives a nonzero polynomial with a Gaussian-null zero set.
Membership is therefore a tail event modulo null sets and has
probability zero or one. See \cite{RosinskiSamorodnitskyTaqqu}.

Square summability of the identity column and of the identity row
each defines such a subspace. If $a\le\beta_m$,
\eqref{eq:lattice-one-sided-lower} bounds
$\norm{P_RT\delta_e}_{L^2(\Omega;\ell^2)}$ away from zero.
The degree-$m$ fourth-moment bound and Paley--Zygmund prevent the
annular column norms from tending to zero in probability. The
column cannot be square summable with probability one, so that
probability is zero. The row excludes $b\le\beta_m$ in the same way.
Sufficiency for the cutoffs follows from the all-order convergence
theorem, and any limit in probability must have the displayed entries.
At or below the local threshold the identity entries have unbounded
variance. Their fourth-moment bound and Paley--Zygmund imply that for
every $K$ there is a probability bounded away from zero that their
supremum over cutoffs exceeds $K$. Bounded sequences form a Borel
linear subspace, so the same zero--one argument gives almost-sure
unboundedness. The fourth-moment bound applied to differences also
excludes convergence in probability.
\end{proof}

\subsection{Exact critical Weyl constants}

Fix $\sigma=D/2$, $a,b>0$, and set $s=a+b$ and
$v=\lim_{R\to\infty}R^{-D}\#B(e,R)$. Put
$S_N=\sum_{|x|\le N}c_x^2\sim vD\log N$ and $R_j=\exp(\sqrt j)$.
The probability measures
\[
 \nu_N(x)=S_N^{-1}c_x^2\mathbf1_{\{|x|\le N\}}
\]
are asymptotically invariant under each fixed left translation. Their
decreasing radial weights express them as mixtures of uniform word balls.
The total mass of balls of bounded radius in this mixture tends to zero,
and the volume asymptotic \eqref{eq:exact-word-growth} makes the uniform
ball measures asymptotically invariant. In particular, the pushforwards
of $\nu_N$ converge to uniform measure on every fixed finite quotient.
For fixed $0<\varepsilon<1<M$, let $P$ be the window projection
onto $\varepsilon R\le\langle x\rangle<MR$, and put
\[
 B_R=R^s S_N^{-m/2}P\Lambda^{-a}\Phi_{m,N}(c)\Lambda^{-b}P,
 \qquad N=\lfloor R/\log R\rfloor .
\]

\begin{lemma}[Critical window reduction]\label{lem:critical-window}
Suppose that for each such window, along $R=R_j$ the singular-value
measures of $B_R$, with mass $R^{-D}$ per singular value in the
window, converge vaguely on $(0,\infty)$ almost surely to the mixture of the law of
$r^{-s}U$ against $vD r^{D-1}\,\mathrm dr$ on $[\varepsilon,M]$,
for a nonnegative variable $U$ with $0<\E U^{D/s}<\infty$.
Then almost surely
\begin{equation}\label{eq:generic-critical-weyl}
 \lim_{n\to\infty}\frac{\mu_n(T)}{n^{-s/D}(\log n)^{m/2}}
 =v^{m/2+s/D}(\E U^{D/s})^{s/D}.
\end{equation}
\end{lemma}

\begin{proof}
Write $\ell_R=\log R$, $t_R=R^{-s}\ell_R^{m/2}$. Work on the sequence
$R=\exp(\sqrt j)$, whose successive ratios tend to one. For fixed
$0<\varepsilon<1<M$, let $P$ project onto the window
$\varepsilon R\le\langle x\rangle<MR$. Deleting rows and columns below
$\varepsilon R$ costs $O(\varepsilon^D R^D)$ in rank. For the terms with
both indices at least $\varepsilon R$ and one at least $MR$, choose an
even $q$ such that $D/q<\min(a,b)$. The almost-sure block bound
\eqref{eq:group-annular-as-upper}, after smoothing, gives for their
sum $E_R$
\begin{equation}
\label{eq:critical-torus-high}
 \limsup_{R\to\infty}\frac{\norm{E_R}_{\Sch_q}}{R^{D/q}t_R}
 \lesssim_{\varepsilon} M^{-\min(a,b)}.
\end{equation}
Indeed, for dyadic output and input scales $L\ge Q$ the sum is bounded by
\[
 \sum_{L\gtrsim MR}L^{-a}(\log(e+L))^{m/2}
 \sum_{\varepsilon R\lesssim Q\le L}Q^{-b+D/q}
 \lesssim (MR)^{-a}(\log(e+MR))^{m/2}(\varepsilon R)^{-b+D/q}.
\]
The order $Q\ge L$ interchanges $a,b$. Restricting the outermost shells
by coordinate projections can only decrease the block norms. Thus the
normalized counting error for $E_R$ above any fixed multiple of $t_R$
tends to zero as $M$ tends to infinity with $\varepsilon$ fixed.

In the window take the low cutoff $N=\lfloor R/\ell_R\rfloor$ and write
$\Phi=\Phi_m(c)$ for the entrywise limit. We have
\begin{equation}
\label{eq:critical-torus-hilbert-tail}
 \E\norm{P(\Phi-\Phi_{m,N}(c))P}_{\Sch_2}^{2}
 \lesssim_M R^D\ell_R^{m-1}\log\ell_R .
\end{equation}
To see this, put $S_u=\sum_{|x|\le u}c_x^2$. The word-ball volume asymptotic
gives $S_u\sim vD\log u$ by summation by parts and
$S_R-S_N=O(\log\ell_R)$. Wick isometry bounds the left side by $m!$
times the squared mass of the ordered difference coefficients. On
$N<\max_i|x_i|\le R$, every tuple has at most $\rank P=O_M(R^D)$
admissible input sites, and the mass of tuples is at most
$m(S_R-S_N)S_R^{m-1}$. On a dyadic largest scale $H\ge R$, fix the input
and output sites and the other $m-1$ slots. The largest slot is then
determined and has squared weight $O(H^{-D})$, while the remaining
weights sum to $O((\log(e+H))^{m-1})$. Summing over two sites and $H$
gives $O_M(R^D\ell_R^{m-1})$. This also proves the bound by passing
from finite upper cutoffs to the entrywise limits. Hilbert-valued
Gaussian hypercontractivity and Markov's inequality, with any fixed
$p>8$, make the probability of exceeding
$R^{D/2}\ell_R^{m/2-1/4}$ summable in $j$. Hence the low cutoff changes
$PTP$ by $o(R^{D/2}t_R)$ in $\Sch_2$ almost surely. Its counting error,
above any fixed positive multiple of $t_R$, is $o(R^D)$.

The $q=2$ annular bound, summed over the bounded number of shells
in this window, and the cutoff error imply $\norm{B_R}_{\Sch_2}
\lesssim_{\varepsilon,M}R^{D/2}$ almost surely. Thus there is no
loss of mass at infinity for the counts above positive thresholds.
By summation by parts, $S_N/\ell_R\to vD$. The radial mixture
has no positive atoms, since $r\mapsto r^{-s}u$ is strictly monotone
for every $u>0$. Hence the window convergence gives counts above
positive thresholds. The singular-value addition inequality with
a fixed positive threshold allowance, the exterior bound
\eqref{eq:critical-torus-high}, and the low rank error pass from
the windows to $T$, first sending $M$ to infinity for fixed
$\varepsilon$ and then sending $\varepsilon$ and the allowance to zero.
Only countably many windows and thresholds are needed. The result is
\begin{align*}
 R^{-D}\#\{n:\mu_n(T)>\lambda t_R\}
 &\longrightarrow vD\int_0^\infty r^{D-1}
       \Prob((vD)^{m/2}U>\lambda r^s)\,\mathrm dr\\
 &=v(vD)^{mD/(2s)}\lambda^{-D/s}\E U^{D/s}.
\end{align*}
Successive radii have ratio tending to one, and the eventually
decreasing $t_R$ have the same property. Monotone inversion with
$\log n\sim D\log R$ proves the formula.
\end{proof}

\begin{theorem}[Critical Weyl constant for virtually abelian groups]
\label{thm:critical-virtually-abelian}
Let $\Gamma$ be infinite, finitely generated and virtually abelian. Take a
free abelian normal subgroup $A$ of finite index $d$, put $D=\rank A$,
and choose representatives $t_f$ for $f\in F=\Gamma/A$, with $t_e=e$.
Write $A$ additively, and set
\[
 B_f h=t_f^{-1}h t_f,\qquad k_{\gamma,f}=t_{\gamma f}^{-1}t_\gamma t_f
 \quad(h\in A,\ \gamma,f\in F).
\]
Choose one representative of each class of the matrices $B_f$ modulo
sign. Let $Z_{\gamma,\eta}$ be independent circular Gaussians with
$\E|Z_{\gamma,\eta}|^2=1/d$, indexed by $\gamma\in F$ and those classes
$\eta$. Let $\chi$ be
Haar distributed on $\widehat A$, independent of these Gaussians. Define
the matrix $G(\chi)$ by
\[
 G(\chi)_{\gamma f,f}=\chi(k_{\gamma,f})
 \begin{cases}
 Z_{\gamma,[B_f]},&B_f\text{ equals its class representative},\\
 \overline{Z_{\gamma,[B_f]}},&B_f\text{ is its negative}.
 \end{cases}
\]
Let $W_m(G(\chi))$ be the degree-$m$ homogeneous Wick projection,
at fixed $\chi$, of $G(\chi)^m$, applied entrywise. Fix $m\ge2$,
$c_x=\langle x\rangle^{-D/2}$ and $a,b>0$, and write $s=a+b$ and
$v=\lim_{R\to\infty}R^{-D}\#B(e,R)$. Then almost surely
\[
 \lim_{n\to\infty}\frac{\mu_n(T_m^{a,b}(c))}
 {n^{-s/D}(\log n)^{m/2}}
 =v^{m/2+s/D}
 \left(\E_{\chi,Z}\operatorname{tr}_d |W_m(G(\chi))|^{D/s}\right)^{s/D}.
\]
Here $\operatorname{tr}_d=d^{-1}\Tr$. In particular,
for $\Gamma=\mathbb Z^D\times F$ with $F$ finite abelian, the normalized
factor is $\Gamma(1+mD/(2s))^{s/D}$ as on the lattice.
\end{theorem}

\begin{proof}
Fix coordinates identifying $A$ with $\mathbb Z^D$ and write each group
element in the form $t_f h$. Left multiplication by $t_\gamma u$
sends $t_f h$ to $t_{\gamma f}(k_{\gamma,f}+B_f u+h)$.
Write $X_N(c)=\Phi_{1,N}(c)$ for the linear Gaussian cutoff and set
$S_N=\sum_{|x|\le N}c_x^2\sim vD\log N$. The Fourier symbol of
$S_N^{-1/2}X_N(c)$ is the matrix $G_N(\chi)$ with entries
\[
 (G_N(\chi))_{\gamma f,f}=\chi(k_{\gamma,f})\sum_{\substack{u\in A\\
 |t_\gamma u|\le N}}\frac{c_{t_\gamma u}}{\sqrt{S_N}}\,
 g_{t_\gamma u}\,\chi(B_f u).
\]
The variance of the sum tends to $1/d$. Indeed, the radial critical
weight is asymptotically invariant, so its masses on the cosets tend
to $1/d$. Distinct coset fields are independent.

We next compute the limiting joint Gaussian law on character grids.
The covariance and pseudo-covariance of two sums in the same coset
use the phases $\chi((B_f-B_{f'})u)$ and $\chi((B_f+B_{f'})u)$,
respectively. For any nonzero integral linear map $L$ on $A$,
an affine fiber $L u=z$ contains $O(H^{D-1})$ points of
$|t_\gamma u|\le H$, uniformly in $z$. Choose a fixed $h_0\in A$
with $Lh_0\ne0$ and translate along its first $O(H)$ powers in $A$.
There are disjoint translates in a ball of radius $O(H)$, which
proves the bound. Consequently the normalized squared coefficient
mass of any fiber is $O(1/S_N)$ by dyadic summation.
On a character grid of side $K$ with $N=o(K)$, different attained
frequencies $Lu$ are distinct modulo $K$ for all large $K$.
For this, their coordinates are $O(N)$: left multiplication by any
generator changes the $A$-coordinate in our coset basis by one
of finitely many shifts. Grid Parseval now bounds the square of the
grid $L^2$ norm of the covariance with $L\ne0$ by the largest
fiber mass times the total mass, hence by $O(1/S_N)$.
When $B_f=B_{f'}$ the sums coincide, and when $B_f=-B_{f'}$ they
are conjugates. These facts give exactly the covariances of $G(\chi)$
in the statement, including the phases.

Take $R=e^{\sqrt j}$, $N=\lfloor R/\log R\rfloor$, and grids of side
$K=\lfloor\delta R\rfloor$ for fixed $\delta>0$. Consider the empirical
law on the grid of the character, the matrix $G_N$, and its real
covariance matrix $Q_N$. The expected law converges weakly to the law
of $(\chi,G(\chi),Q(\chi))$, with $Q(\chi)$ the real covariance of
$G(\chi)$ at fixed $\chi$. This follows from convergence of covariances
in grid mean, continuity of finite-dimensional Gaussian laws, and
Riemann summation of the limiting continuous phases.
For each bounded Lipschitz test, the grid average is a function
of the underlying $g_x$ with Lipschitz constant $O(S_N^{-1/2})$.
Indeed, Cauchy--Schwarz on the grid reduces to the matrix entry
$L^2$ bounds by Parseval, with each $B_f$ permuting the grid.
Gaussian concentration is summable on the chosen radii because
$S_N\asymp\sqrt j$. A countable determining family therefore gives
the empirical weak limit almost surely.

The symbol of $S_N^{-m/2}\Phi_{m,N}(c)$ is the Wick-power matrix of
$G_N$, a polynomial in its entries and their covariances: expand each
matrix entry of the power into paths and use the Wick product on the
linear Gaussian entries along each path. Thus the map taking
$(G_N,Q_N)$ to its singular-value list is continuous, and the grid
empirical singular-value law converges weakly almost surely to
\[
 \nu(E)=\E_{\chi,Z}\operatorname{tr}_d
            \mathbf1_E\bigl(|W_m(G(\chi))|\bigr).
\]
Every positive moment of $\nu$ is finite. It is not concentrated at
zero: on the diagonal path consisting only of the identity coset
letters, a nonzero monomial of degree $m$ in the identity coset
Gaussian appears, and no path using other cosets has that monomial.

It remains to pass from a grid to radial counts. On the spatial
window $\varepsilon R\le\langle x\rangle<MR$, work with the
cutoff matrix in Lemma~\ref{lem:critical-window}. Saturate the window to the full
coset fibers of those $h\in A$ with $\varepsilon R\le\langle h\rangle<MR$.
This costs $o(R^D)$ in rank since $||t_fh|-|h||\le|t_f|$.
Consequently the radial counts in $A$ are asymptotic to
$d^{-1}$ times those in $\Gamma$.
Tile the $A$-coordinates into cubes of side $K$. Their word-length
diameters are $O(\delta R)$ because any lattice displacement of
size $K$ is a product of $O(K)$ fixed lattice generators.
Deleting boundary cubes costs $O_{\varepsilon,M}(\delta R^D)+o(R^D)$
in rank by radial counting. Cutting the matrix between retained cubes
and periodizing the separate cubes costs $o(R^D)$ in rank:
the cutoff operator changes the lattice coordinate by $O_m(N)$,
so the changed rows are within $O_m(N)$ of the cube boundaries.

Within a cube at radius $r$, the input and output weights are equal
to $R^{-b}r^{-b}$ and $R^{-a}r^{-a}$ up to relative $O_{\varepsilon,a,b}(\delta)$.
Freeze these diagonal weights using the singular-value ideal
inequality, also applied to their inverses. The periodic symbols
are the ones considered above. Use compactly supported Lipschitz
tests away from zero. Finite-rank errors are bounded by their rank
times the integral of the absolute derivative of the test, by
singular-value counting. With frozen weights the limiting averages
are continuous functions of $r$, so one can take $\delta$ to zero
and use radial counting in $A$, where the $d$ fiber coordinates
cancel its factor $1/d$. This gives
the limiting singular-value measure of $B_R$, with counts scaled
by $R^{-D}$, as the mixture of the law of $r^{-s}Y$, $Y\sim\nu$, against
$vD r^{D-1}\,{\rm d}r$ on $[\varepsilon,M]$.
This radial mixture has no positive atoms. The result for counts
follows by continuous approximation at each positive threshold.

Lemma~\ref{lem:critical-window} gives the asserted constant. In the direct product with $F$ finite
abelian every $B_f$ is the identity and $k_{\gamma,f}=0$ for the
canonical representatives. The matrix $G$ is finite-group
convolution by independent circular coefficients of variance $1/d$.
Fourier diagonalization on $F$ gives standard circular entries and
the Wick power there is the ordinary power. This proves the last claim.
\end{proof}

\begin{theorem}[Critical Weyl constant beyond the virtually abelian case]
\label{thm:critical-nilpotent-weyl}
Let $\Gamma$ be an infinite finitely generated group of polynomial word
growth degree $D$ which is not virtually abelian, and put
$v=\lim_{R\to\infty}R^{-D}\#B(e,R)$. Fix $m\ge2$,
$c_x=\langle x\rangle^{-D/2}$ and $a,b>0$, and set $s=a+b$.
Let $C=\ell_1+\ell_2^*$ be standard circular on the full Fock space,
with $\ell_i$ the left creation operators for two orthogonal unit
vectors and $\varphi$ the vacuum state. Then almost surely
\[
 \lim_{n\to\infty}\frac{\mu_n(T_m^{a,b}(c))}{n^{-s/D}(\log n)^{m/2}}
 =v^{m/2+s/D}\varphi(|C^m|^{D/s})^{s/D}.
\]
The moments in this formula are given by
\[
 \varphi(|C^m|^{2t})=
 \frac{\Gamma((m+1)t+1)}{\Gamma(t+1)\Gamma(mt+2)}\quad(t\ge0).
\]
\end{theorem}

\begin{proof}
By Gromov's theorem \cite{Gromov} and virtual torsion-freeness, $\Gamma$ has a
torsion-free nilpotent normal subgroup $N_0$ of finite index. It is
nonabelian by hypothesis. We first identify which paired words can
close identically on a product of cosets of $N_0$. Each variable
$x_i$ occurs exactly twice with prescribed signs. Write $x_i=h_i t_i$
on each prescribed coset, with $h_i\in N_0$ and $t_i$ fixed.
If the signed word were an identity there, it would be trivial when
all $h_i=e$. Moving the constants to the right therefore leaves a
word in conjugates of the $h_i^{\pm1}$ by fixed elements of $\Gamma$.

Conjugation extends to automorphisms of the Malcev Lie algebra of
$N_0$. In its quotient by the third lower-central term the bracket
is nonzero, since the algebra is nilpotent and nonabelian.
The signed horizontal maps of the two occurrences of each variable
must be opposites in this quotient: their sum is zero by the assumed
identity and the horizontal images of $N_0$ span the abelianization.
Write the resulting horizontal values as $u_i,-u_i$ in order of
occurrence. Each $u_i$ is an invertible linear transform of the
horizontal coordinate of $h_i$.

If the pairs $i,j$ cross, vary $h_i=g^n,h_j=h^l$ with $g,h\in N_0$
and put the other variables equal to $e$. In exponential coordinates
the central component of their word has an $nl$ term
$\pm[u_i,u_j]$, where $u_i,u_j$ are now the first horizontal values
at $g,h$. Indeed the centered product adds half the bracket, and
the interlaced list $u_i,u_j,-u_i,-u_j$ gives their full bracket.
The one-letter central coordinates are only linear in $n,l$.
Choose $g,h$ with this bracket nonzero by the spanning property.
The word cannot be an identity. If the pairing is noncrossing,
repeatedly erase adjacent opposite-sign pairs, which cancel for
every coset. An adjacent same-sign pair, if encountered, would
force $(h_i t_i)^{2}$ constant as $h_i$ varies in $N_0$,
also when the two signs are negative. Evaluating at the identity
and canceling $t_i^2$ gives
$t_i h_i t_i^{-1}=h_i^{-1}$ on $N_0$. Inversion would then be an
automorphism, impossible in a nonabelian group. This argument
does not depend on the order of the coset in the finite quotient.
Thus the only surviving coset identities are the noncrossing
opposite-sign pairings, which are identities on the full group.

Give each pair the law $\nu_N$. By the equation theorem for virtually
nilpotent groups \cite[Theorem~1.13]{MartinoTointonValiunasVentura},
solutions outside the identity cosets are negligible by finite
quotients of the product group. Their product sampling probability
tends to zero because any finite quotient of a power of $\Gamma$
factors through a power of a finite quotient of $\Gamma$.

Consequently, in the normalized trace moments of the Wick polynomial
only the noncrossing opposite-sign pairings survive. Each alternating
factor $\Phi_{m,N}$ or $\Phi_{m,N}^*$ consists of $m$ slots of the same
sign, so every surviving pair respects the Wick prohibition on internal
pairing. Their count is $\varphi(|C^m|^{2k})$ for the trace with $2k$
factors, by the last-in-first-out rule for creation and annihilation
on Fock space.

For the window $\varepsilon R\le\langle x\rangle<MR$ with projection
$P$, take $R=e^{\sqrt j}$, $N=\lfloor R/\log R\rfloor$ and
$B_R=R^s S_N^{-m/2}P\Lambda^{-a}\Phi_{m,N}\Lambda^{-b}P$.
The radial coordinate along every trace path changes by at most
$O_{m,k}(N)=o(R)$, since $\bigl||gx|-|x|\bigr|\le|g|$.
Except at $o(R^D)$ starting sites near the radial boundaries, all the
spatial restrictions are automatic and the weights differ from
$(\langle x\rangle/R)^{-2ks}$ by $o(1)$ uniformly. Boundary
contributions are negligible since each pairing has probability at
most one under $\nu_N$ and the normalized weights are bounded.
Radial counting and the paired-word limit yield
\[
 R^{-D}\E\Tr((B_R^*B_R)^k)
 \longrightarrow vD\,\varphi(|C^m|^{2k})
            \int_\varepsilon^M r^{D-1-2ks}\,{\rm d}r .
\]
In the variance, only pairings joining the two traces remain after
centering. Such a pair occurs once in the first trace and is fixed
by its word equation, which costs at most $1/S_N$.
Summing the two starting sites shows that the normalized trace
variance is $O(1/S_N)$. Fixed-degree hypercontractivity and
Borel--Cantelli give almost-sure convergence of all trace moments.
The limiting squared singular-value measure is the mixture of the
vacuum law of $r^{-2s}|C^m|^2$ against $vD r^{D-1}\,{\rm d}r$ on
$[\varepsilon,M]$. It is compactly supported, so its moments determine
it, and the radial variable removes all positive atoms.

The vacuum distribution of $|C^m|$ has positive second moment,
since $C^m\Omega=\ell_1^m\Omega$. Lemma~\ref{lem:critical-window}
therefore gives the Weyl coefficient with this law for $U$.

For integral $t$ the moment formula is the Fuss--Catalan pairing count
\cite[Corollary~3.2]{KempSpeicher}. By Gauss's multiplication formula
the displayed gamma ratio is the $t$th moment of
$\frac{(m+1)^{m+1}}{m^m}\prod_{i=1}^m B_i$ for all $t\ge0$, where
the independent beta variables have parameters $i/(m+1)$ and
$(i+1)/m-i/(m+1)$. Integer moments determine both compact laws.
\end{proof}

\begin{corollary}[The critical Weyl constant detects virtual abelianness]
For a group $\Gamma$ as in the critical Weyl theorems, let $m=2$ and
$a,b>0$ with $a+b=D/2$. Let $d_c(\Gamma)$ be the limit of the probability
that two independent elements with law
$\nu_N(x)=S_N^{-1}c_x^2\mathbf1_{\{|x|\le N\}}$ commute. Then almost surely
\[
 \lim_{n\to\infty}\frac{\mu_n(T_2^{a,b}(c))}
                  {n^{-1/2}\log n}
       =v^{3/2}\bigl(1+d_c(\Gamma)\bigr)^{1/2}.
\]
The right side is strictly greater than $v^{3/2}$ exactly when
$\Gamma$ is virtually abelian.
\end{corollary}

\begin{proof}
The limiting commuting probability exists by the equation theorem
\cite[Theorem~1.13]{MartinoTointonValiunasVentura} and asymptotic
invariance of $\nu_N$. At $m=2$, in the normalized Hilbert--Schmidt
trace there are two Wick pairings between an adjoint word and an
ordinary word. One closes identically by inverse cancellation, and
the other closes precisely when its two labels commute. Thus the
second moment of the limiting normalized singular law is
$1+d_c(\Gamma)$. In the virtually abelian matrix model this is
also $\E_{\chi,Z}\operatorname{tr}_d|W_2(G(\chi))|^2$:
the finite Gram matrices converge in grid mean and their
Gaussian polynomial moments are uniformly bounded, so the
second moment converges along with the singular law.
Formula~\eqref{eq:generic-critical-weyl} gives the result.
For a nonvirtually abelian group the commuting equation has no
identity coset by the crossing argument, so its probability limit
vanishes. In the virtually abelian case two samples in the abelian
subgroup of index $d$ commute, giving a lower bound $d^{-2}$.
\end{proof}

For the infinite dihedral group, $d_c=1/4$: two rotations commute, and every other commuting coset pair meets an affine constraint on the rotation coordinates. The volume-normalized constant in the corollary is then $\sqrt{5/4}$.

\section{Finite-group convolution spectra}

\subsection{Limits of projected Plancherel measures}

For a finite group $G$, let $\widehat G$ denote its irreducible complex
representations up to equivalence. Its projected Plancherel measure is
\[
 \mu_G=\frac1{|G|}\sum_{\pi\in\widehat G}(\dim\pi)^2\,\delta_{\dim\pi}
\]
on the one-point compactification $\N\cup\{\infty\}$ of the positive
integers. Adamczak asks which weak limits of these measures can occur
\cite[Section~7(I)]{Adamczak}.

\begin{theorem}[Limits of projected Plancherel measures]
\label{thm:projected-plancherel-classification}
The weak limits of $\mu_{G_i}$ along finite groups with $|G_i|\to\infty$
are $\delta_\infty$ and the measures
\begin{equation}\label{eq:plancherel-limit-form}
 \frac1n(M_k)_*\mu_H+\left(1-\frac1n\right)\delta_\infty ,
 \qquad n,k\in\N,\quad H\text{ a finite group},\quad M_k(d)=kd .
\end{equation}
\end{theorem}

\begin{proof}
Suppose $\mu_{G_i}\to\mu\ne\delta_\infty$. The commuting probability of
$G_i$ is $|\widehat G_i|/|G_i|=\int d^{-2}\,\mathrm d\mu_{G_i}(d)$,
where $\infty^{-2}=0$, and hence stays bounded below. Neumann's theorem
\cite[Theorem~2.4]{Eberhard} gives normal subgroups $K_i\triangleleft G_i$
of class at most two with $[G_i:K_i]$ and $|K_i'|$ bounded.
On a subsequence identify $K_i'$ with a fixed group $C$ and take
the conjugation images of $G_i$ in $\operatorname{Aut}(C)$ to coincide.

For $L\le C$ put $R_i(L)=\{x\in K_i:[x,K_i]\subset L\}$.
After passing to a subsequence every index $[K_i:R_i(L)]$ stays
bounded or tends to infinity. The subgroups $L$ giving bounded
indices include $C$ and are closed under intersection because
$R_i(L\cap L')=R_i(L)\cap R_i(L')$, as in
\cite[Lemma~3.2]{Browning}. Their least member $L_0$ is invariant
under the common conjugation image, so it is normal in each $G_i$.

An irreducible of $K_i$ with central character $\chi$ on $C$ has
degree squared $[K_i:R_i(\ker\chi)]$. Conjugation multiplies the matrix
of $x$ by the commutator character, so its trace vanishes off
$R_i(\ker\chi)$. On this subgroup the matrix is scalar by Schur's
lemma, and character orthogonality gives the degree formula.
All constituents on restriction of a bounded-degree irreducible
of $G_i$ to $K_i$ have bounded degree. By the index dichotomy their
central characters eventually kill $L_0$. Thus for $Q_i=G_i/L_0$
and $n=|L_0|$, each fixed $d$ eventually satisfies
$\mu_{G_i}(\{d\})=n^{-1}\mu_{Q_i}(\{d\})$.
The center of $K_i/L_0$ is an abelian normal subgroup $A_i$ of $Q_i$
of bounded index.

Pass to a subsequence with fixed quotient $F=Q_i/A_i$.
The quotient acts on $A_i$. Pass further so that each $|(f-1)A_i|$
is bounded or tends to infinity. The elements of bounded image size
form a normal subgroup $J\le F$, since
$|(fg-1)A_i|\le |(f-1)A_i|\,|(g-1)A_i|$ and these sizes are
invariant under inversion and conjugation. Let $P_i$ be its
preimage in $Q_i$ and put $k=[F:J]$. For a uniform
$\chi\in\widehat A_i$, the probability that $f$ fixes $\chi$ is
$1/|(f-1)A_i|$. Hence the stabilizer of $\chi$ is contained in $J$
with probability tending to one.

For $A\triangleleft Q$ abelian, the degree law $\mu_Q$ can be sampled
by choosing $\chi\in\widehat A$ uniformly and then choosing
$\pi\in\widehat Q$ with conditional probability
\begin{equation}\label{eq:plancherel-conditional}
 \frac{(\dim\pi)m_\pi(\chi)}{[Q:A]},
\end{equation}
where $m_\pi(\chi)$ is its multiplicity on restriction to $A$.
Restriction of the regular representation gives the normalization.
Averaging over $\chi$ gives the Plancherel weight. If $\chi$ has
stabilizer contained in $J$, the $A_i$-weights of a $P_i$-irreducible
above it form one $J$-orbit. Conjugates from other cosets of $P_i$
have disjoint orbits. By Frobenius reciprocity, induction to $Q_i$
is then irreducible and is a bijection on the irreducibles above
this fixed $\chi$. The disjoint orbits preserve the $\chi$-multiplicity.
The degree and the denominator in \eqref{eq:plancherel-conditional}
both multiply by $k$, so the coupling by $\chi$ gives
$\norm{\mu_{Q_i}-(M_k)_*\mu_{P_i}}_{\rm TV}\to0$.

The measures $\mu_{P_i}$ have only finitely many possible values.
Indeed $[P_i,A_i]$ has bounded order since it is the product of the
subgroups $(f-1)A_i$ for $f\in J$. The image of $A_i$ is central of
bounded index in $P_i/[P_i,A_i]$. Schur's theorem
\cite[Lemma~2.5]{Eberhard} then gives a bound on $|P_i'|$.
The characters of $A_i$ extending to $P_i$ form
$D_i=(A_i\cap P_i')^\perp$, of bounded index in $\widehat A_i$,
since characters extend from subgroups of finite abelian groups.
Twisting by these extensions shows that the conditional
degree law \eqref{eq:plancherel-conditional} for $P_i$ is constant
on each coset of $D_i$. Every degree is at most $[P_i:A_i]=|J|$
by Frobenius reciprocity with the induction of some $A_i$-character.
Thus the probabilities occur in multiples of
$1/(|J|[\widehat A_i:D_i])$, with bounded denominator and bounded
support. After another subsequence $\mu_{P_i}=\mu_H$ for one finite
group $H$. The finite atoms of $\mu$ are now given by
\eqref{eq:plancherel-limit-form}.

For the converse, $\delta_k$ is the limit for
$\mathbb F_q^+\rtimes C_k$ with $q\to\infty$, $k\mid q-1$, and
$C_k$ acting by scalars of order $k$. Such $q$ run through powers
of a fixed prime coprime to $k$. The nontrivial additive characters
have trivial stabilizer, so the degree law is
$q^{-1}\delta_1+(1-q^{-1})\delta_k$.
For $n\ge2$ use the Heisenberg group on
$(\mathbb Z/n\mathbb Z)^r\times(\mathbb Z/n\mathbb Z)^r
\times\mathbb Z/n\mathbb Z$ with central cocycle $x\cdot y'$.
Under Plancherel measure the character of its center is uniform
by restriction of the regular representation. For a central
character of order $d$, the radical of the commutator form has
index $d^{2r}$, since the $2r$ coordinates must be multiples of
$d$. The degree is $d^r$ by the class-two formula above.
As $r\to\infty$ the projected law tends to
$n^{-1}\delta_1+(1-n^{-1})\delta_\infty$. For $n=1$ use the
trivial group, and taking $n$ prime tending to infinity with
$r=1$ gives $\delta_\infty$. Direct products multiply independent
Plancherel degrees. Combining the two constructions with $H$
gives all the limits in \eqref{eq:plancherel-limit-form}.
\end{proof}

Adamczak's Theorems~1.5 and~1.10 \cite{Adamczak} transfer this
classification to the singular-value and complex Gaussian eigenvalue
limits under their respective coefficient hypotheses. Replace
$\delta_d$ by the expected empirical law of the corresponding
$d\times d$ Ginibre spectrum with entry variance $1/d$. For
$\delta_\infty$ use the quarter-circle or circular law, respectively.
Compactness of the projected measures gives the same list for all
subsequential spectral limits, and every mixture on the list occurs.

\subsection{Finite-group convolutions with general covariance}
\label{subsec:convolution-covariance}

Adamczak asks in \cite[Section~7(III)]{Adamczak} for the analogue of his
singular-value and Gaussian eigenvalue theorems when the entry
pseudo-variance is $\alpha\in[0,1]$. The circular case $\alpha=0$ is
in \cite{Adamczak}, and the real endpoint $\alpha=1$ is treated by
Gerspach \cite[Theorems~6 and~7]{Gerspach}. A rotation of the
coefficients makes any complex pseudo-variance nonnegative. It
rotates eigenvalues and preserves singular values.

For an irreducible representation $\pi$ of a finite group $G$, write
$d_\pi=\dim_{\mathbb C}\pi$ and let its type
$\mathbb F_\pi\in\{\mathbb R,\mathbb C,\mathbb H\}$ be given by its
Frobenius--Schur indicator. The type-resolved Plancherel subprobability
measures are
\[
 \mu_{G,\mathbb F}=|G|^{-1}
       \sum_{\pi:\mathbb F_\pi=\mathbb F}d_\pi^2\delta_{d_\pi},
 \qquad \mu_G=\sum_{\mathbb F}\mu_{G,\mathbb F}.
\]
Write $W_d^{\mathbb R}$ and $W_d^{\mathbb C}$ for the real and circular
complex $d$-by-$d$ Ginibre matrices with entry variance $1/d$. For
even $d=2h$, define the quaternionic model in complex coordinates by
\[
 W_d^{\mathbb H}=\begin{pmatrix}B&C\\-\overline C&\overline B\end{pmatrix},
\]
where $B,C$ are independent $h$-by-$h$ circular complex matrices with
independent entries of variance $1/d$. For each type let
\[
 W_{d,\alpha}^{\mathbb F}
  =\sqrt\alpha\,W_d^{\mathbb F}
    +\sqrt{1-\alpha}\,\widetilde W_d^{\mathbb C},
\]
with the two summands independent. Let $\rho_{d,\alpha}^{\mathbb F}$
and $\theta_{d,\alpha}^{\mathbb F}$ be its expected empirical
singular-value and eigenvalue measures, respectively. At $d=\infty$
use the quarter-circle and circular laws. Let
$D_{\mathbb R}=D_{\mathbb C}=\N\cup\{\infty\}$ and
$D_{\mathbb H}=2\N\cup\{\infty\}$ be the admissible degree sets.

\begin{theorem}[Convolution spectra with general covariance]
\label{thm:convolution-covariance}
Let $|G_i|\to\infty$ and let the $\xi_g$ be independent copies of a
centered variable $\xi$ with $\E|\xi|^2=1$ and $\E\xi^2=\alpha\in[0,1]$.
Write $\nu_i$ for the empirical singular-value measure of
$|G_i|^{-1/2}\sum_{g\in G_i}\xi_g\lambda_g$. Then
\[
 d_{\rm w}\!\left(\nu_i,\,
 \sum_{\mathbb F}\sum_{d\in D_{\mathbb F}}
    \mu_{G_i,\mathbb F}(\{d\})\,\rho_{d,\alpha}^{\mathbb F}\right)
 \xrightarrow{\ \mathbb P\ }0 ,
\]
where $d_{\rm w}$ is any metric for weak convergence. If $\xi$ is
jointly Gaussian in its real
and imaginary parts, the same holds for the empirical eigenvalue
measure with $\theta_{d,\alpha}^{\mathbb F}$ in place of
$\rho_{d,\alpha}^{\mathbb F}$. When the three type-resolved measures
converge weakly, replacing $\mu_{G_i,\mathbb F}$ by their limits gives
the corresponding weak limit in probability.
\end{theorem}

\begin{proof}
First suppose the three type-resolved measures converge.
For Gaussian coefficients write $\xi_g=\sqrt\alpha\,u_g+
\sqrt{1-\alpha}\,v_g$, with independent standard real $u_g$ and circular
complex $v_g$. The Fourier transform of the convolution is the direct
sum of Fourier blocks, each repeated $d_\pi$ times. The real and
circular Gaussian transforms have respective marginals
$W_{d_\pi}^{\mathbb F_\pi}$ and $W_{d_\pi}^{\mathbb C}$
\cite[Proposition~3.3]{Adamczak}\cite[Lemma~2]{Gerspach}.
Schur orthogonality gives zero ordinary cross-covariances for distinct
representations and zero cross-pseudo-covariances off conjugate pairs.
Thus the conjugation orbits are independent. Within a block the
ordinary entry covariance is $d_\pi^{-1}I$. Interleaving the two spin
coordinates gives independent $2$-by-$2$ cells in the quaternionic type.

As $d\to\infty$, the empirical eigenvalue law of $W=W_{d,\alpha}^{\mathbb F}$
converges to the circular law. In the real and complex representation
types the entries are i.i.d., so \cite[Theorem~1.1]{NguyenORourke}
applies. For quaternionic
type with $d=2h$, the i.i.d. cells of $\sqrt h\,W$ satisfy Condition
C0 of \cite[Definition~2.6]{NguyenORourke}. Their entries are centered
with variance $1/2$ and zero ordinary cross-covariance. Apply their
Theorem~2.7 with zero deterministic perturbation.

For the expected singular-value measures, expand
$d^{-1}\E\Tr((WW^*)^k)$ by Wick's rule. A pair identifies the ordered
cells of its entries, reversing indices on adjoints. The resulting
quotient of the trace walk has at most $k$ unoriented edges and $k+1$
vertices. With $k+1$ vertices it is a tree, and each edge is traversed
twice in opposite directions. A pair of two $W$ entries or two $W^*$ entries
therefore allows at most $k$ free base indices. Its contribution is
$O_k(d^{-1})$ after trace normalization, since there are at most two
spin choices per index and each covariance is $O(d^{-1})$.
The ordinary pairings match circular Ginibre exactly. Hence the
squared singular-value moments converge to the Marchenko--Pastur
moments, which determine the law, and $\rho_{d,\alpha}^{\mathbb F}$
converges to the quarter-circle law.

The expected Gaussian empirical laws thus converge to the stated
mixtures. For eigenvalues, the weights of conjugation orbits of degree
at most $M$ are at most $2M^2/|G_i|$. Their weighted sum against a
bounded continuous test has variance $O(M^2/|G_i|)$. For higher degrees
replace each test integral by the circular one, with mean absolute
error tending to zero uniformly as $M\to\infty$.

For singular values we also compare arbitrary input with its Gaussian
counterpart. Write $N=|G_i|$, $A=N^{-1/2}\sum_g \xi_g\lambda_g$, and
let $D=\begin{psmallmatrix}0&A\\A^*&0\end{psmallmatrix}$. For
$\operatorname{Im}w>0$, put $R=(D-wI)^{-1}$ and $f=(2N)^{-1}\Tr R$.
Changing one complex coefficient by $v$ changes $D$ by a self-adjoint
direction $U$ of operator norm $|v|/\sqrt N$. For $f$ and $|f|^2$,
directional derivatives of orders $j\le3$ are bounded by
$C_w\|U\|_{\mathrm{op}}^j$ by the resolvent identity and the product
rule. The second-order Taylor remainder is bounded by
$C_w\min(|v|^3N^{-3/2},|v|^2N^{-1})$, using either the third derivative
or the second derivative along the segment. In successive Gaussian
replacements the matching two-dimensional covariance cancels the
quadratic terms, and centering cancels the linear terms. The total
remainder tends to zero by dominated convergence of
$\E[|\xi|^2\min(|\xi|/\sqrt N,1)]$ for each input law.

For Gaussian input the real and imaginary coefficient directions are
Hilbert--Schmidt orthogonal with norm $\sqrt2$, by orthogonality of the
regular shifts. For each direction $U$, $\partial_U f=-(2N)^{-1}
\Tr(UR^2)$. Bessel's inequality bounds the squared gradient by
$2(2N)^{-2}\|R^2\|_{\mathrm{HS}}^2=O_w(N^{-1})$.
Gaussian Poincar\'e therefore gives variance $O_w(N^{-1})$.
The comparisons of $\E f$ and $\E|f|^2$ give convergence in probability
to the same deterministic transform for arbitrary input.
The normalized eigenvalue measure of $D$ is the symmetrized
singular-value measure of $A$. Its second moment has expectation one
by regular-shift orthogonality. Tightness and convergence of Stieltjes
transforms on a countable dense set yield the claimed singular-value
convergence. For a general group sequence, every subsequence admits a
further subsequence on which all three typed measures converge. The
empirical laws and finite-group mixtures then have the same weak limit.
This proves the approximation in probability along the original
sequence.
\end{proof}

\section*{Use of generative AI}

The author completed this paper with the assistance of OpenAI's GPT
models. The author takes full responsibility for the correctness of
the paper.

\begingroup
\renewcommand{\bibliofont}{\normalsize}

\endgroup

\end{document}